\documentclass[12pt,twoside]{article}

\usepackage[T1]{fontenc}
\usepackage[cp1250]{inputenc}

\usepackage[arrow, matrix, curve]{xy}
\usepackage{amsmath,amsthm,amssymb}
\usepackage{amscd,graphics}
\usepackage[dvips]{graphicx}
\usepackage{makeidx}

\newenvironment{Proof}{\textbf{Proof.}}{$\qquad \blacksquare$\par}
\newenvironment{Proof of}[1]{\textbf{Proof #1.}}{$\qquad \blacksquare$\par}
\newenvironment{Item}[1]{\par #1 }{\par}

\DeclareMathOperator{\Aut}{Aut}
\DeclareMathOperator{\clsp}{\overline{span}}
\newcommand{\LL}{\mathcal L}
\newcommand{\TT}{\mathcal T}
\newcommand{\JJ}{\mathcal J}
\newcommand{\KK}{\mathcal K}
\newcommand{\OO}{\mathcal{O}}

\renewcommand{\L}{\mathcal L}

\newcommand{\FF}{\mathcal F}
\newcommand{\A}{ A} 
\newcommand{\B}{B}

\newcommand{\Z}{\mathbb Z}
\newcommand{\N}{\mathbb N}
\newcommand{\T}{\mathbb T}

\renewcommand{\H}{\mathcal H}
\newcommand{\I}{\mathcal I}

\newcommand{\M}{\mathcal M}

\newcommand{\NN}{\mathcal N}

\newcommand{\X}{\widetilde X}

\newcommand{\tdelta}{\widetilde \alpha}

\newcommand{\al}{\alpha}

\renewcommand{\L}{\mathcal L}

\newtheorem{thm}{Theorem}[section]
\newtheorem{lem}[thm]{Lemma}
\newtheorem{prop}[thm]{Proposition}
\newtheorem{cor}[thm]{Corollary}
\theoremstyle{definition}

\newtheorem{defn}[thm]{Definition}
\newtheorem{ex}[thm]{Example}
\newtheorem{rem}[thm]{Remark}

 \makeindex

\begin{document}

 \thispagestyle{empty}

 \begin{center}
{\bfseries \large \textsc{Crossed products by endomorphisms  and  reduction of relations in relative Cuntz-Pimsner algebras}}

\bigskip
B. K.  Kwa\'sniewski\footnote{Work partially supported by National Science Centre  grants numbers DEC-2011/01/D/ST1/04112 and DEC-2011/01/B/ST1/03838} , \ \  A. V. Lebedev

\end{center}

\begin{abstract} Starting
from an arbitrary endomorphism $\alpha$ of a unital $C^*$-algebra
$\A$ we construct a crossed product. It is shown that the natural
construction depends not only on the $C^*$-dynamical system
$(\A,\alpha )$ but also on the choice of an ideal $J$ orthogonal
to $\ker \alpha $. The article gives an  explicit description of
the internal structure of this crossed product and, in particular,
discusses the interrelation between relative Cuntz-Pimsner
algebras and  partial isometric crossed products. We present a
canonical procedure that reduces any given $C^*$-correspondence
 to the  'smallest' $C^*$-correspondence  yielding the
same relative Cuntz-Pimsner algebra as the initial one. In the
context of crossed products this reduction procedure corresponds
to the reduction of $C^*$-dynamical systems and allow us to
establish  a coincidence between relative Cuntz-Pimsner
algebras and crossed products introduced.
\end{abstract}
\medbreak

 \textbf{Keywords:} \emph{$C^*$-algebra, endomorphism, partial
isometry, orthogonal ideal, crossed product, covariant representation, $C^*$-correspondence, relative Cuntz-Pimsner algebra, reduction}

\medbreak
{\bfseries 2000 Mathematics Subject Classification:} 47L65, 46L05,
47L30
\vspace{5mm}

\tableofcontents

\section*{Introduction}

The crossed product of a $C^*$-algebra $A$ by an automorphism
$\alpha:A\to A$ is defined as a universal $C^*$-algebra
generated by a copy of $A$ and a unitary element $U$ satisfying
the relations
$$
\alpha(a)= Ua U^*,\qquad \alpha^{-1}(a)=U^*aU, \ \ \ \ a\in A.
$$
On  one hand, algebras arising in this way (or their versions adapted to actions
of groups of automorphisms) are very well understood and  became
a part of a $C^*$-folklore \cite{Pedersen}, \cite{Kadison}. On the other hand,
it is  very symptomatic  that, even though  the first attempts on  generalizing
this kind  of constructions to endomorphisms  go  back to 1970s,  articles introducing
different definitions of the related object appear almost  continuously until the present-day, see,
for example, \cite{CK}, \cite{Paschke},
\cite{Stacey}, \cite{Murphy},  \cite{exel1}, \cite{exel2},
\cite{kwa}, \cite{Ant-Bakht-Leb}.
This phenomenon is caused by   very fundamental problems one has to face when dealing with crossed
products by endomorphisms. Namely, one has to answer the following questions:

\begin{itemize}
    \item[(i)] What relations should the element  $U$  satisfy?
        \item[(ii)] What should be used in place of $\alpha^{-1}$?
\end{itemize}

It is important that in spite of  substantial freedom of
choice (in answering the foregoing questions), all the above
listed papers  do however have a certain nontrivial intersection.
They  mostly agree, and simultaneously boast their greatest
successes, in the  case when dynamics is implemented by
monomorphisms with  a hereditary range. In view of the  articles
\cite{Bakht-Leb}, \cite{Ant-Bakht-Leb}, \cite{kwa3}, this
coincidence is completely understood. It is shown in
\cite{Bakht-Leb} that in the case of monomorphism with hereditary
range there exists a unique non-degenerate transfer operator
$\alpha_*$ for $(A,\alpha)$, called by authors of
\cite{Bakht-Leb} a \emph{complete transfer operator}, and the
theory goes smooth with  $\alpha_*$ as it takes over the role
classically played by $\alpha^{-1}$. The $C^*$-dynamical systems
of this sort will be  called {\em partially reversible}.
\par
If a pair $(A,\alpha)$ is of the above described type,
then $A$ is called a {\em coefficient algebra}.
This notion was introduced in \cite{Leb-Odz} where its investigation and
relation to the extensions of $C^*$-algebras by partial isometries was clarified.
Further in
 \cite{Bakht-Leb} a certain criterion for a
$C^*$-algebra to be a coefficient algebra associated with a given
endomorphism was obtained (see also  \cite{kwa3}). On the base of the results of these papers
one naturally arrives at the construction of a certain crossed
product which was implemented in \cite{Ant-Bakht-Leb}.
It was also observed in \cite{Ant-Bakht-Leb} that in the most
natural situations the coefficient algebras arise as a result of
a certain extension procedure on the initial $C^*$-algebra.
Since the crossed product is (should be) an extension of the
initial $C^*$-algebra one can consider the construction of an
appropriate coefficient algebra as one of the most important
intermediate steps in the procedure of construction of the crossed product itself (a detailed discussion of the philosophy of the arising construction is given in \cite[Section~5]{Ant-Bakht-Leb}).

It was also shown in \cite[Section~4]{Ant-Bakht-Leb} how  different extension procedures lead to the most popular constructions of crossed products such as Cuntz-Krieger algebras~\cite{CK}, Paschke's crossed product~\cite{Paschke}, partial crossed product~\cite{exel1}, Exel's crossed product~\cite{exel2} and others.

The analysis of  these extension procedures   naturally  leads to  the next  problem: can we extend a
 $C^*$-dynamical system associated with an arbitrary endomorphism
to a {\em partially reversible}  $C^*$-dynamical system? In the
{\em commutative} $C^*$-algebra situation the corresponding
procedure and the explicit description of maximal ideals of the
arising $C^*$-algebra is given in \cite{maxid}. On the base of
this construction the general construction of the crossed product
associated to an arbitrary endomorphism of a {\em commutative}
$C^*$-algebra is presented in \cite{kwa}. Further the general
construction of an extension of a $C^*$-dynamical system
associated with an arbitrary endomorphism to a  partially
reversible $C^*$-dynamical system   is worked out in
 \cite{kwa4}. Therefore the mentioned results of
\cite{Bakht-Leb}, \cite{Ant-Bakht-Leb},  \cite{kwa4}, give us the key to
construct a general crossed product starting from a
$C^*$-dynamical system associated with an arbitrary endomorphism,
and this is one of the main  themes of the present article.
\par
The most important novelties we incorporate to the theory of crossed products are
\begin{itemize}
    \item[1)] an explicit description of the crossed product based on the
worked out matrix calculus presented in Section~\ref{22},
\end{itemize}
and an observation of (in a way unexpected) phenomena  that
\begin{itemize}
\item[2)] in general the universal construction of the crossed product
depends not only on the algebra $A$ and an endomorphism $\alpha$
one starts with but also on the choice of an ({\em arbitrary})
singled out ideal $J$ orthogonal to the kernel of $\alpha$ (see
Section \ref{Crossed products-alternative1}).
\end{itemize}
So in fact we have a {\em variety} of crossed products depending
on $J$.
\smallskip

On the appearance of  \cite{kwa-leb} B. Solel noted to the authors that the crossed product constructed
in \cite{kwa-leb} can  also be modeled as a certain relative Cuntz-Pimsner algebra (Proposition \ref{universality proposition}
of the present article describes in essence the main idea of B. Solel's remark). Thus we have naturally arrived at the discussion of  interrelations between   relative Cuntz-Pimsner algebras and  crossed products, and this was the theme of~\cite{kwa-leb1}. Since relative Cuntz-Pimsner algebras are defined by means of $C^*$-correspondences the role of the latter objects in the whole picture should  be clarified and in this way we necessarily come to the analysis of the interplay: crossed products -- relative Cuntz-Pimsner algebras -- $C^*$-correspondences.
Corollary \ref{C-P-cross} of the present article states that if $X$ is a $C^*$-correspondence of a $C^*$-dynamical system $(A,\alpha)$ and $J$ is an ideal orthogonal to the kernel of $\alpha$ then  the relative Cuntz-Pimsner algebra
$\OO(J,X)$  and the crossed product $C^*(A,\alpha,J)$ of the present article are canonically isomorphic. This observation in its turn  causes a problem. Namely,  $\OO(J,X)$ is defined for ideals $J$ that are not
necessarily orthogonal to the kernel of $\alpha$. Moreover, by means of $\OO(J,X)$ one can construct crossed products seemingly different from those introduced in the present article (see, for example, Stacey's crossed product identified in Corollary \ref{corollary stacey's}). Therefore, one may   guess  that
$\OO(J,X)$ is a more general object than  $C^*(A,\alpha,J)$. At
the same time it is known (see \cite[Proposition~2.21]{ms}, i.e. Proposition~\ref{injectivity of k_A} of the present article)    that when
$J$ is not  orthogonal to the kernel of $\alpha$  the algebra  $\OO(J,X)$
possesses certain 'degeneracy'. All this stimulates us to take a
closer look and provide a more thorough analysis of the structure
of $\OO(J,X)$ and its relation to $C^*(A,\alpha,J)$, and this is one more main goal of the article.

As we show the necessary apparatus of investigation of the noted  vagueness in the relation between $\OO(J,X)$ and $C^*(A,\alpha,J)$   is \emph{reduction}.

The general scheme of reduction procedure (taking quotients) associated with ideals  in $C^*$-correspondences and the corresponding reduction in relative Cuntz-Pimsner algebras as well as the analysis of this scheme was provided in the structure theorem of \cite{fmr} (see Theorem \ref{takie tam aa} of the present article). We add  \emph{canonicity} to this scheme by applying the reduction procedure to a sequence of ideals $J_n, \ n=1,2, ..., \, J_\infty$  (Definition \ref{reduction ideal for correspondences}) that are naturally generated by $J$ and $\OO(J,X)$. Being defined in this way the  canonical
procedure  reduces any given  $C^*$-correspondence to  the  'smallest' $C^*$-correspondence yielding the  same relative Cuntz-Pimsner algebra as the initial one. In the context of the crossed products this reduction procedure corresponds to the reduction of $C^*$-dynamical systems.
 Using this, on the one hand,  we  obtain the \emph{canonical $C^*$-dynamical system} (in Section \ref{Canonical C*-dynamical systems}) and, on the other hand,  eliminate the mentioned 'degeneracy' in $\OO(J,X)$ and simultaneously establish an isomorphism between $\OO(J,X)$ and appropriate crossed product introduced in the present article (Theorem~\ref{reduction thm} and Proposition~\ref{reducing C*-Hilbert bimodules})  obviating in this way the mentioned vagueness in their interrelations.  As a byproduct we also get  a refinement of Stacey's results  (Example \ref{reduction of Stacey's crossed product}) and add clarity to Katsura's canonical relations (Subsection \ref{Katsura's canonical relations}).

 We would also like to make a certain additional remark on the objects of the paper. It is generally agreed that a crossed product of a \emph{unital $C^*$-algebra} $A$ by a $C^*$-mapping should be a $C^*$-algebra $B$ generated by a copy (or at least a homomorphic image) of the $C^*$-algebra $A$ and an operator $U$ implementing the dynamics. On the other hand when $A$ is \emph{non-unital}, then it seems that there are various essentially different ideas of what the \emph{non-unital crossed product} $B$ should be. In particular, one can add to questions (i), (ii) (related to irreversibility of dynamics)  two more still open principle problems (related to the lack of unity):
\begin{itemize}
    \item[(iii)] What should $B$ be generated by? Should it be the set $A\cup AU$, cf. e.g. \cite{Pedersen}, \cite{Kadison}, \cite{Stacey}, \cite{brv}, or maybe $A\cup A_0U$ for a certain subspace $A_0$ of $A$, cf. \cite{exel1},  \cite{fmr}, if so what should $A_0$ be?
        \item[(iv)] How should $U$ be related to $B$? Should it belong to the multiplier algebra $M(B)$ of $B$, cf. \cite{Pedersen}, \cite{Kadison}, \cite{Stacey},  an enveloping $W^*$-algebra $B^{**}$ of $B$, see e.g.  \cite{exel1}, \cite{brv}, or maybe something else?
\end{itemize}
A way to bypass this trouble, usually  adopted by most of the authors, cf. e.g.  \cite{Lin-Rae}, \cite{fmr}, \cite{brv},  is to consider the non-unital crossed products only for the so-called \emph{extendible systems}, that is for systems which naturally extend from $A$ to the multiplier algebra $M(A)$. For  such systems a non-unital crossed product is actually a subalgebra of a unital crossed product.  In  the present paper we drop the technicalities arising from consideration of  extendible systems and  non-trivial issues concerning non-extendible systems. \emph{We consider only unital crossed products}. Nevertheless the general $C^*$-correspondences will be considered over arbitrary (not necessarily unital) $C^*$-algebras.

The present paper is based on \cite{kwa-leb} and ~\cite{kwa-leb1},
and in essence forms their unification, refinement and development.

The paper is organized as follows.

In the first section we discuss and clarify the relations that should be used in a definition of a covariant representation. In particular, we split the class of covariant representations
 into subclasses according to a certain ideal they determine arriving at the notion of a $J$-covariant representation.
In Subsection \ref{what da hell} we establish  existence of such representations and introduce the corresponding crossed products as universal algebras. Section~\ref{22} presents a matrix calculus  which  describes the internal algebraic structure
 of the crossed product serving simultaneously as its  certain regular representation. This leads us in Subsection \ref{norm} to an explicit formula for the norm of  elements of the crossed product introduced, thus providing  us with one more its alternative definition (Definition~\ref{cr-pr-def}).
In Section \ref{isomorph} we give a series of isomorphism theorems. They  provide an apparatus for verifying faithfulness of a given representation of crossed product and, in particular,  establishing equivalence of different approaches to construction of
crossed products (cf. Proposition~\ref{crossed equiv}). The isomorphism theorems are discussed as on the operator algebraic level so also on the dynamical topological level exploiting topological freeness
of the arising $C^*$-dynamical systems (Theorems \ref{isomorphiasm theorem} and \ref{topolo disco-polo}). In addition we present here an overview of existing crossed product constructions and their comparison with the crossed product  of the present paper.  This stimulates us to undertake deeper analysis of interrelation between different approaches  and in this way we pass to the next main theme of the article.
In Section~\ref{C-P}, Subsections \ref{preliminaries on C-correspondences}, \ref{The Toeplitz C*-algebra of a Hilbert bimodule},~\ref{Relative Cuntz-Pimsner algebras}, we   recall the indispensable notions and objects  concerning $C^*$-correspondences and  Cuntz-Pimsner algebras, while
 presentation of the crossed products as  relative Cuntz-Pimsner algebras is given in
Subsection~\ref{Crossed prod =Relative Cuntz-Pimsner algebras}. Analysis of interrelations:
$C^*$-correspondences -- relative
Cuntz-Pimsner
algebras -- crossed products is implemented on the base of reduction.  The principal results in this direction  are given in Section~\ref{Reduction and canonical}. Applying the canonical reduction to $C^*$-correspondences we
eliminate the 'degeneracy' in relative Cuntz-Pimsner algebras   (Theorem \ref{reduction thm}). Applying this  procedure to  $C^*$-dynamical systems we prove coincidence between the corresponding  relative Cuntz-Pimsner algebras  and the crossed products of the present article (Proposition~\ref{reducing C*-Hilbert bimodules}).  Finally in  Section~\ref{Canonical C*-dynamical systems} starting from a triple $(A,\al,J)$ we construct a $C^*$-dynamical system $(A_J,\al_J)$ which is cannonical in the sense that the corresponding relations defining  the crossed product are  'non-degenerate' and do not include any ideal of $A$ (Theorem~\ref{canon}). This canonical construction is related to but as we argue differs from the similar result due to Katsura.

\section{Covariant representations, orthogonal ideals, crossed product}\label{Crossed products-alternative1}

\subsection{Covariant representations and ideals of covariance} \label{start}


Let $(\A,\alpha)$ be a   pair consisting of a $C^*$-algebra $\A$, containing an identity and
an endomorphism $\alpha : \A \to \A$ (by a  homomorphism between $C^*$-algebras we always mean a $^*$-homomorphism). Throughout the paper the pair $(\A,\alpha)$ will be called
a {\em $C^*$-dynamical system}.

\begin{defn}\label{kowariant  rep defn*}
Let $(\A,\alpha)$ be a $C^*$-dynamical system. A
\emph{representation} of  $(\A,\alpha)$ is a  triple $(\pi,U,H)$
consisting of a unital  representation $\pi:\A\to L(H)$ on a
Hilbert space $H$ and an operator $U\in L(H)$ satisfying  the
following relation
\begin{equation}\label{covariance rel1*}
U\pi(a)U^* =\pi(\alpha(a)),\qquad a \in A.
\end{equation}
If $\pi$ is a faithful representation of $\A$, then $(\pi,U,H)$ is
called a \emph{faithful representation}.
 \end{defn}
 Since $\al^n(1)$ is a projection for every $n$ it follows that the operator $U$ in the above definition is necessarily  a \emph{power partial isometry}.

 Note also that iterating \eqref{covariance rel1*} we get
 \begin{equation}\label{n}
 U^{n}\pi(a)U^{*n} =\pi(\alpha^n(a)),\qquad a \in \A, \, n\in \N
 \end{equation}
 which means that
  if     $(\pi,U,H)$ is a representation of   $(\A,\alpha)$, then  $(\pi,U^n,H)$ is a representation of $(\A,\alpha^n)$ for every $n\in \N$.

 The next lemma shows that representations of $C^*$-dynamical systems possess one more important property which plays, in fact, a crucial role in the whole story.
 \begin{lem}\label{iteration of representations}
  If    $(\pi,U,H)$ is a representation of   $(\A,\alpha)$, then  for every $n\in \N$ we have
 \begin{equation}\label{covariance rel2*}
U^{*n}U^n \in \pi(A)'.
\end{equation}.
\end{lem}
\begin{Proof}
Let $(\pi,U,H)$ consists of a unital  representation $\pi:\A\to L(H)$ on a
Hilbert space $H$ and an operator $U\in L(H)$ such that \eqref{covariance rel1*} holds. In view of \eqref{n} it suffice to prove \eqref{covariance rel2*} only for $n=1$. As  $U$ is a partial isometry $1-U^*U$ is a projection, and for all $a \in A$
\begin{align*}
\|U\pi(a)(1-U^*U)\|^2&=\|U\pi(a)(1-U^*U)\pi(a^*)U^*\|
\\
&=\|U\pi(a)\pi(a^*)U^* - U\pi(a)U^*U\pi(a^*)U^*\|
\\
&=\|\pi(\al(aa^*))-\pi(\al(a)\al(a^*))\|=0.
\end{align*}
Thus $U\pi(a)=U\pi(a)U^*U=\pi(\al(a))U$ and  by passing to adjoints we also get $U^*\pi(\al(a))=\pi(a)U^*$. Using these two relations we get
$$
U^*U\pi(a)=U^*\pi(\al(a))U=\pi(a)U^*U,
$$
which proves the  assertion.
\end{Proof}

  \begin{rem}\label{remarks on representations} \begin{itemize}
  \item[1.] The   notion  of  a representation  of a $C^*$-dynamical system
 appears  in a similar or identical form, for instance,  in \cite{Stacey},  \cite{Adji_Laca_Nilsen_Raeburn}, \cite{Murphy}, \cite{exel2}, \cite{Lin-Rae}, \cite{Leb-Odz}, \cite{kwa}, \cite{kwa4}. For isometric crossed products, cf. \cite{Stacey},  \cite{Adji_Laca_Nilsen_Raeburn}, \cite{Murphy}, it is assumed that $U$ is an isometry satisfying \eqref{covariance rel1*}.
  In general, cf. \cite{exel2},\cite{Lin-Rae}, \cite{Leb-Odz}, \cite{kwa}, \cite{kwa4}, definitions of representations of $C^*$-dynamical systems contained conditions \eqref{covariance rel1*} and \eqref{covariance rel2*} (for $n=1$) or a certain equivalent of \eqref{covariance rel2*}. Lemma \ref{iteration of representations} shows that \eqref{covariance rel2*} is redundant.

\item[2.] In  \cite[Prop 2.2]{Leb-Odz} and  \cite[Lem. 4.3]{Lin-Rae} it was shown that when  \eqref{covariance rel1*} is assumed relation \eqref{covariance rel2*} is equivalent to the condition
\begin{equation}\label{covariance rel1**}
U\pi(a) =\pi(\alpha(a))U,\qquad a \in A.
\end{equation}
In view of Lemma \ref{iteration of representations}    the conditions \eqref{covariance rel2*} and \eqref{covariance rel1**} are not only equivalent in the presence of \eqref{covariance rel1*} but they actually follow from \eqref{covariance rel1*}.
\end{itemize}
\end{rem}

Relation \eqref{covariance rel2*} imply that for   any  representation
$(\pi,U,H)$ of $(A,\al)$ the set
$$
J:=\{ a\in \A: U^*U \pi(a)=\pi(a)\}$$ is  an ideal in $\A$ (by which we  always mean a closed two-sided ideal).
This ideal will play one of the key roles in the paper. The next statement shows a certain property of $J$ which
is important for the further analysis.

\begin{prop}\label{motivation prop1}
Let $(\pi,U,H)$ be a  representation of $(\A,\alpha)$ and let
\begin{equation}\label{ideals I and J}
I=\{ a\in \A: (1-U^*U) \pi(a)=\pi(a)\}, \qquad J=\{ a\in \A: U^*U \pi(a)=\pi(a)\}.
\end{equation}
Then
$$
I= \ker (\pi\circ\alpha)
 \qquad
\textrm{ and }\qquad I\cap J = \ker\pi.
$$
\end{prop}
\begin{Proof}
Observe that $
U (U^*U \pi(a))U^* = U  \pi(a)U^*
$
and
$
U^* (U \pi(a)U^*)U =U^*U \pi(a).
$
Therefore the mappings
$
U(\cdot )U^*:U^*U \pi(a)\mapsto\pi(\alpha(a)),
$
$
U^*(\cdot
)U:\pi(\alpha(a))\mapsto U^*U \pi(a)
$
are each other inverses and in particular
 $U^*U \pi(\A)\cong \pi(\alpha(\A))$, where $U^*U \pi(a)\mapsto\pi(\alpha(a))$. Thus
$$
 \pi(\alpha(a))=0
\Longleftrightarrow U^*U\pi(a)=0 \Longleftrightarrow
(1-U^*U)\pi(a)=\pi(a).
$$
Which means that   $I$ is the kernel of $\pi\circ\alpha$.

Clearly $\ker\pi\subset I\cap J$ and on the other hand
$$
a\in  I\cap J \Rightarrow (1-U^*U) \pi(a)= U^*U \pi(a)\Rightarrow \pi(a)=0,
$$
that is $ I\cap J \subset\ker\pi$.
\end{Proof}
\begin{cor}\label{ort-id}
Let $(\pi,U,H)$ be a  faithful   representation of $(\A,\alpha)$, and let $I$ and $J$ be ideals in $A$ given by \eqref{ideals I and J}. Then
$$
I= {\ker}\, \alpha \qquad and \qquad I\cap J = \{0\}.
$$
\end{cor}
Having in mind this observation we introduce the following
\begin{defn}
\label{ort}
Let $I$ and $J$ be two ideals in $\A$. We say that $J$ is \emph{orthogonal to $I$} if
$$
I\cap J =\{0\}.
$$
There exists the biggest ideal orthogonal  to $I$ (in the sense
that it contains all other ideals that are orthogonal to $I$)  denoted by $I^\bot$. This ideal  could be defined
explicitly as $I^\bot=\{a\in \A:aI=\{0\}\}$ and it is called the
\emph{annihilator} of the ideal $I$ in~$\A$.
\end{defn}

Proposition~\ref{motivation prop1} and Corollary~\ref{ort-id} make it natural to introduce the following definition.

\begin{defn}\label{kowariant  rep defn 2} Let $(\pi,U,H)$ be a  representation
 of $(\A,\alpha)$ and let $J$  be an ideal in $\A$.
If   $(\pi,U,H)$  and  $J$ are such that
\begin{equation}\label{covariance rel3}
J=\{ a\in \A: U^*U \pi(a)=\pi(a)\},
\end{equation}
then we will say that
 $(\pi,U,H)$  is a  $J$-\emph{covariant representation}.  In this situation we will also say  that $J$ is the
 \emph{ideal of covariance  for the  representation} $(\pi,U,H)$. If $J = (\ker\alpha)^\bot$, then we simply call $(\pi,U,H)$ a \emph{covariant representation}.
\end{defn}

\begin{rem}\label{remark something1}  For any $C^*$-dynamical system $(A,\al)$  one can always
construct a faithful  representation of $(\A,\alpha)$ (see e.g. the Toeplitz representation defined
in Subsection  \ref{what da hell}) such that $U$ is not an isometry. 
 However, if $(\pi,U,H)$ is a
 covariant representation of $(\A,\alpha)$, then the following implications hold true:
\begin{itemize}
\item[1)] $\alpha$ is a monomorphism $\Rightarrow$ $U$ is an isometry,
\item[2)] $\alpha$ is an automorphism  $\Rightarrow$ $U$ is unitary.
\end{itemize}
Hence, in contrast to arbitrary representations, covariant representations as defined above involve the operators typically used in similar definitions for automorphisms and monomorphisms of $C^*$-algebras, cf. \cite{Pedersen}, \cite{Kadison},  \cite{Paschke},
\cite{Stacey}, \cite{Murphy}.
\end{rem}

In connection with the above remark and for future applications we recall the following observation, which is a part a) of  \cite[Prop. 1.9]{kwa4}, and  follows immediately from Proposition \ref{motivation prop1}.
\begin{prop}\label{proposition najwazniejsze}
Suppose $(\A,\alpha)$ is such that $\ker\alpha$ is unital and let $(\pi,U,H)$ be a representation of $(\A,\alpha)$. Then the following conditions are equivalent:
\begin{itemize}
\item[i)] $(\pi,U,H)$ is a   covariant representation
\item[ii)] $U^*U\in \pi(\A)$
\item[iii)] $U^*U\in \pi(Z(\A))$ ($Z(\A)$ stands for the center of $\A$)
\item[iv)] $U^*U$ is the unit in $\pi((\ker\alpha)^\bot)$
\end{itemize}
\end{prop}
Though we have defined $J$-covariant representations, their existence for all approriate ideals $J$ is not established yet. Henceforth we present a construction resolving this problem.

\subsection{Construction of a faithful $J$-covariant representation. Crossed product} \label{what da hell}
Let us  fix a $C^*$-dynamical system $(\A,\alpha)$, an ideal
$J\subset(\ker\alpha)^\bot$, and a faithful nondegenerate
representation $\pi:\A\to L(H)$.  First we define a  triple
$(\widetilde{\pi},\widetilde{U}, \H)$ by the formulae
$$
\H:=\bigoplus_{n=0}^{\infty} \pi(\alpha^n(1))H, \qquad
(\widetilde{\pi}(a)h)_n:=\pi(\alpha^n(a))h_n,\qquad
(\widetilde{U}h)_n:= h_{n+1}.
$$
One readily sees that $(\widetilde{\pi},\widetilde{U}, \H)$ is a
faithful  representation of $(\A,\alpha)$. Actually $(\widetilde{\pi},\widetilde{U}, \H)$ is $\{0\}$-covariant  and having in mind the
classical associations one could call
$(\widetilde{\pi},\widetilde{U}, \H)$ a \emph{Toeplitz
representation} of $(\A,\alpha)$.

In order to obtain a $J$-covariant representation we introduce the
following algebra of operators
$$
c_0(\N,J):=\{a=\bigoplus_{n\in \N}\pi(a_n): a_n \in
\alpha^n(1)J\alpha^n(1),\,\,\, \lim_{n\to \infty} a_n =0\} \subset L(\H),
$$
and consider the $C^*$-algebra $C^*(c_0(\N,J),\widetilde{U})$ generated by $c_0(\N,J)$ and $\widetilde{U}$. One checks that
$$
\widetilde{U}c_0(\N,J) \widetilde{U}^*\subset c_0(\N,J) , \qquad  \widetilde{U}^*c_0(\N,J) \widetilde{U}\subset c_0(\N,J), \qquad \widetilde{U}^* \widetilde{U} \in c_0(\N,J)',
$$
and hence, see \cite[Prop 2.3]{Leb-Odz}, $C^*(c_0(\N,J),\widetilde{U})$ is the closure of elements of the form
$$
\widetilde{U}^{*n}a^{(-n)} + ...+ \widetilde{U}^*a^{(-1)}+a^{(0)} + a^{(1)}\widetilde{U} + ... +a^{(n)}\widetilde{U}^n,
$$
where $ a^{(k)} \in c_0(\N,J), k=0,\pm1,...,\pm n$. Thus using the relations
$$
\widetilde{\pi}(A)c_0(\N,J)\subset c_0(\N,J), \qquad \widetilde{\pi}(A)\widetilde{U}^*=\widetilde{U}^*\widetilde{\pi}(\al(A)), \qquad \widetilde{U}\widetilde{\pi}(A)=\widetilde{\pi}(\al(A))\widetilde{U}
$$
one sees that $C^*(c_0(\N,J),\widetilde{U})$ is an  ideal in $C^*(\widetilde{\pi}(\A),\widetilde{U})$. Let us now take any faithful non-degenerate representation of the quotient algebra
$$
\widetilde{\pi}_J:C^*(\widetilde{\pi}(\A),\widetilde{U})/C^*(c_0(\N,J),\widetilde{U})\to L(H_J)
$$
and put
$$
U_J:=\widetilde{\pi}_J([\widetilde{U}]), \qquad
\pi_J (a):=\widetilde{\pi}_J ([\widetilde{\pi}(a)]), \ a\in\A,
$$
where $
[\,\cdot\,]:C^*(\widetilde{\pi}(\A),\widetilde{U})\to C^*(\widetilde{\pi}(\A),\widetilde{U})/C^*(c_0(\N,J),\widetilde{U})
$ is the quotient mapping.
\begin{prop}\label{regular representation proposition}
For any ideal $J\subset(\ker\alpha)^\bot$, the triple $(\pi_J,U_J,
H_J)$ defined above is a faithful $J$-covariant representation
of  $(\A,\alpha)$.
\end{prop}
\begin{Proof}
That
$(\pi_J,U_J, H_J)$ is a  representation of  $(\A,\alpha)$ is straightforward. What we
need to verify is that it is faithful and $J$-covariant.

To see that $\pi_J$ is
faithful let $a\in \ker \pi_J\subset c_0(\N,J)$. Then $\alpha^{n}(a)\in J$ for all $n=0,1,2...$, and $\pi(\alpha^n(a))\to
0$. Since $J\cap\ker\alpha=\{0\}$ the homomorphism $\alpha: J\to \A$ is isometric,
and as $\pi$ is a faithful representation of $\A$  we
get
$$
\|a\|=\lim_{n\to \infty}\|\pi(\alpha^n(a))\|=0.
$$
To see that $\pi_J$ is $J$-covariant note that for any $a\in \A$
$$
\widetilde{U}^*\widetilde{U}\widetilde{\pi} (a) - \widetilde{\pi} (a)= \{\pi(a),0,0,... \},
$$
and so, by the definition of $c_0(\N,J)$
$$
[\widetilde{U}^*\widetilde{U}\widetilde{\pi} (a) - \widetilde{\pi} (a)] = [0]\ \Leftrightarrow
 \ a\in J.
$$
\end{Proof}
As an immediate corollary we have
\begin{thm}\label{existance theorem}
Let $(\A,\alpha)$ be a $C^*$-dynamical system and  $J$ be   an ideal in $\A$.
Then there exists a   faithful $J$-covariant representation iff \ $\ker\alpha\cap J=\{0\}$.
\end{thm}
\begin{Proof} Sufficiency follows from Proposition~\ref{regular representation proposition}
while Corollary~\ref{ort-id} implies necessity.
\end{Proof}
The foregoing observations make   the next definition of the crossed product natural.
\begin{defn}\label{crossed product defn}
Let $(\A,\alpha)$ be a $C^*$-dynamical system and $J$ an ideal in $\A$ such that $J\cap \ker \alpha=\{0\}$. The \emph{crossed product} $
C^*(\A,\alpha,J)
$ of $\A$ by $\alpha$ associated with $J$ is   a universal $C^*$-algebra generated by the the copy of the algebra $\A$ and a partial isometry $u$ subject to relations
$$
ua u^*=\alpha(a),\,\,\,\qquad
J=\{a\in \A: u^*u a=a\}.
$$
\end{defn}

The aim of the paper is the analysis of the crossed product
introduced. It will be shown  that this construction covers the
main  known  constructions of crossed products. In addition we
present a thorough description  of the internal structure of this
crossed product (Sections~\ref{22},~\ref{isomorph}) and discuss
its relation to the relative Cuntz-Pimsner algebras
(Section~\ref{C-P}). In particular, by means of the reduction
procedure (Section~\ref{Reduction and canonical}) we show that all
relative Cuntz-Pimsner algebras are associated with orthogonal ideals.

\section{Crossed product and matrix calculus}
\label{22}
This section presents an alternative
definition of the crossed product which provides us with one more
interesting, integral point of view leading to a transparent
description of its internal structure. One can also consider the construction described here as a sort of regular representation of the crossed product.

\subsection{Matrix calculus}
\label{2}

We introduce a  matrix calculus that will be an algebraic framework for  our  crossed product.

Let us denote by $\M(\A)$  the set of infinite matrices
$\{a_{i,j}\}_{i,j\in\N} $  indexed by pairs of
natural numbers  with entries $a_{i,j}$ in $\A$  such that
$$ a_{i,j}\in \alpha^i(1)\A \alpha^j(1), \qquad i, j\in \N,$$
and there is at most finite number of $a_{i,j}$ which are non-zero.

We will take advantage of this standard matrix notation when
defining operations on $\M(\A)$ and investigating a natural
homomorphism from $\M(\A)$ to the covariance algebras generated by
 representations of $(A, \alpha)$. However,  when calculating norms of elements
in covariance algebras, it is more handy to index the entries of
an element in $\M(\A)$ by a pair consisting of a natural number
and an integer. Hence  we will paralellely use two notations
concerning matrices in $\M(\A)$. Namely, we presume the following
identifications
 $$
a_{i,j}=a^{(j-i)}_{\min\{i,j\}},\quad i,j \in \N,\quad  \qquad a_{n}^{(k)}=\begin{cases}
a_{n,k+n},& k\geq 0\\
a_{n-k,n},& k < 0
\end{cases}\quad n \in \N,\,\, k\in \Z,
$$
under which we have two equivalent matrix presentations
$$
 \left(
\begin{array}{cc c c}
  a_{0,0}      &  a_{0,1}   &  a_{0,2}  &   \cdots  \\
  a_{1,0}     &  a_{1,1}   &  a_{1,2}  &   \cdots  \\
  a_{2,0}     &  a_{2,1}  &  a_{2,2}  &   \cdots  \\
    \vdots       &   \vdots     &   \vdots    &    \ddots
\end{array}\right)
=
\left(
\begin{array}{cc c c}
  a_0^{(0)}      &  a_0^{(1)}   &  a_0^{(2)}  &   \cdots  \\
  a_0^{(-1)}     &  a_1^{(0)}   &  a_1^{(1)}  &   \cdots  \\
  a_0^{(-2)}     &  a_1^{(-1)}  &  a_2^{(0)}  &   \cdots  \\
    \vdots       &   \vdots     &   \vdots    &    \ddots
\end{array}\right).
$$
The use of each of these   conventions  will always be clear from the context.

 We define the
addition, multiplication by scalar, and
involution on $\M(\A)$  in a natural manner. Namely, for $a=\{a_{ij}\}_{i,j\in\N}$ and $b=\{b_{ij}\}_{i,j\in\N}$ in $\M(\A)$
we put
\begin{gather}\label{add1}
 (a+b)_{m,n}=a_{m,n}+b_{m,n},\\[6pt]
\label{mulscal1}
 (\lambda a)_{m,n}=\lambda a_{m,n}\\[6pt]
\label{invol1}
 (a^*)_{m,n}=a_{n,m}^*.
\end{gather}
 Moreover,  we introduce a convolution multiplication '$\star$' on
$\M(\A)$, which is a reflection of the operator multiplication in covariance algebras.
 We set
\begin{equation}\label{star1}
a\star b= a\cdot\sum_{j=0}^\infty  \Lambda^j(b)+ \sum_{j=1}^\infty \Lambda^j(a)\cdot b
\end{equation}
where $\cdot$ is the standard multiplication of matrices and  mapping  $\Lambda:
\M(\A)
\rightarrow \M(\A)$ is defined
 to act as follows: $\Lambda(a)_{i,j}=\alpha(a_{i-1,j-1})$, for $i,j> 0$, and $\Lambda(a)_{i,j}=0$ otherwise,
 that is $\Lambda$ assumes the following shape
\begin{equation}\label{Lambda}
\Lambda(a)= \left(
\begin{array}{cc c c c }
 0      &         0           &           0         &            0       & \cdots  \\
 0      &  \alpha(a_{0,0})  &  \alpha(a_{0,1})  &  \alpha(a_{0,2}) & \cdots  \\
 0      &  \alpha(a_{1,0}) &  \alpha(a_{1,1})  &  \alpha(a_{1,2}) & \cdots  \\
 0      &  \alpha(a_{2,0}) &  \alpha(a_{2,1}) &  \alpha(a_{2,2}) & \cdots  \\
\vdots  &       \vdots        &       \vdots        &       \vdots       & \ddots
\end{array}\right).
\end{equation}
\begin{prop}
The set $\M(\A)$ with operations \eqref{add1},
\eqref{mulscal1}, \eqref{invol1}, \eqref{star1} becomes an algebra with involution.
\end{prop}
\begin{Proof}
The only thing we show is  associativity of  multiplication \eqref{star1}, the rest is straightforward.
For that purpose we note that $\Lambda$ preserves the standard matrix multiplication and thus we have
$$
a\star (b \star c)= a\star  \Big(b\cdot\sum_{j=0}^\infty  \Lambda^j(c)+ \sum_{j=1}^\infty \Lambda^j(b)\cdot c\Big)
$$
$$
=a\sum_{k=0}^\infty  \Lambda^k\Big( b\sum_{j=0}^\infty  \Lambda^j(c)+
\sum_{j=1}^\infty \Lambda^j(b) c \Big)+ \sum_{k=1}^\infty \Lambda^k(a) \Big( b\sum_{j=0}^\infty  \Lambda^j(c)+
\sum_{j=1}^\infty \Lambda^j(b) c
\Big)$$
$$
=\sum_{k,j=0}^\infty  a \Lambda^k( b)\Lambda^{j}(c) +  \sum_{k=1,j=0}^\infty \Lambda^{k}(a) b\cdot \Lambda^j(c) +
\sum_{k,j=1}^\infty \Lambda^{k}(a)\Lambda^{j}(b) c
$$
$$
=\Big(a\sum_{k=0}^\infty  \Lambda^k(b)+ \sum_{k=1}^\infty \Lambda^k(a) b\Big)\sum_{j=0}^\infty \Lambda^j(c) +
\sum_{j=1}^\infty \Lambda^j \Big(a\sum_{k=0}^\infty  \Lambda^k(b)+ \sum_{k=1}^\infty \Lambda^k(a) b\Big) c
$$
$$
  \Big(a\cdot\sum_{k=0}^\infty  \Lambda^k(b)+ \sum_{k=1}^\infty \Lambda^k(a)\cdot b\Big)\star c= (a\star b) \star c.
$$
\end{Proof}
We embed $\A$ into $\M(\A)$ by identifying an element $a\in \A$ with
the matrix $\{a_{m,n}\}_{m,n\in \N}$ where $a_{0,0}=a$ and $a_{m,n}=0$
if $(m,n)\neq (0,0)$.  We also define a 'partial isometry'
$u=\{u_{m,n}\}_{m,n\in \N}$ in $\M(\A)$ such that
$u_{0,1}=\alpha(1)$ and $u_{m,n}=0$ if $(m,n)\neq (0,1)$. In other
words, we adopt the following notation
\begin{equation}
\label{notation stupid} u= \left(
\begin{array}{cc c c}
  0      &  \alpha(1)   &  0  &   \cdots  \\
  0     &  0   &  0  &   \cdots  \\
  0     & 0  &  0  &   \cdots  \\
    \vdots       &   \vdots     &   \vdots    &    \ddots
\end{array}\right)\quad \mathrm{and \,\,} \quad
a=\left(\begin{array}{cc c c}
  a      &  0     &  0    & \cdots  \\
  0      &  0 &  0    & \cdots  \\
  0     &  0    &  0   & \cdots  \\
     \vdots                 &    \vdots                &      \vdots             &\ddots
\end{array}\right), \quad \mathrm{for \,\,} a\in \A.
\end{equation}
 One can  check that the $^*$-algebra $\M(\A)$ is generated
by $u$ and $\A$. Furthermore, for every $a\in \A$ we have
$$
 u\star a \star u^* =\alpha(a) \quad\textrm{ and } \quad  u^*\star a \star u = \left(
\begin{array}{cc c c}
  0      &  0    &  0  &   \cdots  \\
  0     &  \alpha(1)a\alpha(1)   &  0  &   \cdots  \\
  0     & 0  &  0  &   \cdots  \\
    \vdots       &   \vdots     &   \vdots    &    \ddots
\end{array}\right).
 $$
 \begin{prop}\label{dense subalgera structure prop}
Let $(\pi,U,H)$ be  a  representation of $(\A,\alpha)$. Then there
exists a  unique $^*$-homomorphism $\Psi_{(\pi, U)}$ from $\M(\A)$
onto a $^*$-algebra $C_0^*(\pi(\A),U)$ generated by $\pi(\A)$ and
$U$, such that
$$
\Psi_{(\pi, U)}(a)=\pi(a),\quad a\in \A,\qquad \Psi_{(\pi, U)}(u)=U.
$$
Moreover, $\Psi_{(\pi, U)}$ is given by the formula
\begin{equation}\label{Psi form eq}
\Psi_{(\pi, U)}( \{a_{m,n}\}_{m,n\in\N})=\sum_{m,n=0}^\infty U^{*m}\pi(a_{m,n}) U^n,
\end{equation}
and thus $C_0^*(\pi(\A),U)=\left\{\sum_{m,n=0}^\infty U^{*m}\pi(a_{m,n}) U^n: \{a_{m,n}\}_{m,n\in\N}\in \M(\A)\right\}$.
\end{prop}
\begin{Proof} It is clear that $\Psi_{(\pi, U)}$ has to satisfy    \eqref{Psi form eq}.
Thus it is enough to check that $\Psi_{(\pi, U)}$ is a $^*$-homomorphism, and in fact we only need to show
that $\Psi_{(\pi, U)}$ is multiplicative as the rest is obvious. For that purpose
let us fix two matrices $a=\{a_{m,n}\}_{m,n\in\N},b=\{b_{m,n}\}_{m,n\in\N}\in \M(\A)$. We will examine the product
$$
c_{p,r,s,t}=U^{*p}\pi(a_{p,r})U^{r}  U^{*s}\pi(b_{s,t} )U^{t}
$$
Depending on the relationship between $r$ and $s$ we have two cases.
\begin{Item}{1)}
 If $s \leq r$, then using equality $(U^s U^{*s})\pi(b_{s,t})=\pi(\al^s(1)b_{s,t})=\pi(b_{s,t}) $ and relation $U^{*r-s} U^{r-s}\in \pi(A)' $, see Lemma \ref{iteration of representations}, we get
$$
c_{p,r,s,t}=U^{*p}\pi(a_{p,r})U^{r-s} (U^s U^{*s})\pi(b_{s,t}) U^t=
U^{*p}\pi(a_{p,r})U^{r-s} U^{*r-s} U^{r-s}\pi(b_{s,t}) U^t
$$
$$
=U^{*p}\pi(a_{p,r})U^{r-s} \pi(b_{s,t})(U^{*r-s} U^{r-s}) U^t=U^{*p}\pi(a_{p,r} \alpha^{r-s}(b_{s,t})) U^{t+r-s}
$$
Putting $r-s=j$, $r=i$, $p=m$ and $t+r-s =n$ we get
$$
c_{m,r,s,n-r+s}=c_{p,r,s,t}=U^m \pi(a_{m,i} \alpha^{j}(b_{i-j,n-j})) U^{n}
$$
and thus
$$
\sum_{s,r \in \N \atop s \leq r}  c_{m,r,s,n-r+s}=
\sum_{j=0}^\infty\sum_{i=j}^\infty  U^m \pi(a_{m,i} \alpha^{j}(b_{i-j,n-j})) U^{n}=
U^m \pi(\big(a\cdot\sum_{j=0}^\infty  \Lambda^j(b)\big)_{m,n})  U^{n}
$$
\end{Item}
\begin{Item}{2)}
If $r <s$, then analogously
$$
c_{p,r,s,t}=U^{*p} \pi(a_{p,r})(U^r  U^{*r}) U^{*s-r}\pi(b_{s,t}) U^l=
U^{*p}\pi( a_{p,r})(U^{*s-r}U^{r-s})U^{*r-s}\pi(b_{s,t}) U^t
$$
$$
=U^{*p} (U^{*s-r}U^{s-r})\pi(a_{p,r})U^{*s-r}\pi(b_{s,t} )U^t=U^{*p+s-r} \pi(\alpha^{s-r}(a_{p,r})b_{s,t}) U^t
$$
Putting $s-r=j$, $r=i$, $p +s-r=m$ and $t=n$ we get
$$
c_{m-s+r,r,s,n}=c_{p,r,s,t}=U^m\pi( \alpha^{j}(a_{m-j,i-j})b_{i,n} ) U^{n}
$$
and thus
$$
\sum_{s,r \in \N \atop r< s}  c_{m-s+r,r,s,n}=
\sum_{j=1}^\infty\sum_{i=j}^\infty  U^m \pi(\alpha^{j}(a_{m-j,i-j})b_{i,n} ) U^{n}=
U^m \pi(\big(\sum_{j=1}^\infty  \Lambda^j(a)\cdot b\big)_{m,n})  U^{n}
$$
\end{Item}
\noindent Using the formulas obtained in 1) and 2)  we have
$$
\Psi_{(\pi, U)}(a) \Psi_{(\pi, U)}( b)= \sum_{p,r,s,t\in \N} c_{p,r,s,t}=
\sum_{p,r,s,t\in \N \atop s\leq r}c_{p,r,s,t} + \sum_{p,r,s,t\in \N \atop r<s} c_{p,r,s,t}$$
$$
=\sum_{m,r,s,n\in \N \atop s\leq r,\,  n\leq r-s}c_{m,r,s,n-r+s} +
\sum_{m,r,s,n\in \N \atop r<s,\,m\leq s-r} c_{m-s+r,r,s,n}=
\sum_{m,n\in \N} U^{*m} (a\star b)_{m,n} U^n =\Psi_{(\pi, U)}(a\star b)
$$
and the proof is complete.
\end{Proof}

 We now examine the structure of $\M(\A)$.
    We will say that  a matrix $\{a_{n}^{(m)}\}_{n\in\N, m\in \Z}$ in $\M(\A)$  is  $k$-\emph{diagonal},
    where $k$ is an integer, if it satisfies the condition
$$
a_{n}^{(m)}\neq 0  \Longrightarrow\,\, m=k.
$$
In other words $k$-diagonal matrix is the one of the form
\begin{center}\setlength{\unitlength}{1mm}
\begin{picture}(110,24)(-5,-10)

\put(-10,0){$
 \left(\begin{array}{c}\begin{xy}
\xymatrix@C=-1pt@R=3pt{
   &      \, \,   &    \qquad \,\,0 \,   \\
     &      \,     &       \\
      &     \, 0      &       \\
     &       &   \qquad
        }
  \end{xy}
  \end{array}\right)
 $  }
  \put(3.5,11){\scriptsize $k$}
  \qbezier[10](2, 10)(4,10)(6, 10)

 \qbezier[42](2, 10)(9,2.5)(16,-5)
  \qbezier[36](6, 10)(11,4.5)(16,-1)

   \put(29,0){if $k\geq 0$, or}
  \put(98,0){if $k< 0$.}
 \put(61,0){$
 \left(\begin{array}{c}\begin{xy}
\xymatrix@C=-1pt@R=4pt{
   &      \, \,   &    \,   \\
     &      \,     &  0      \\
      &     \,      &       \\
   0 \,\, &       &   \qquad
        }
  \end{xy}
  \end{array}\right)$}
  \put(58,3.6){\scriptsize $|k|$}
  \qbezier[10](64,2) (64,4) (64,5.5)
  \qbezier[42](64,5.5)(72.4,-1.5)(80.8,-8)
  \qbezier[36](64,2)(70,-3.5)(76,-8)
     \end{picture}
     \end{center}
  The  linear space consisting of all $k$-diagonal matrices will be denoted by $\mathcal{M}_{k}$.
  These spaces will correspond to spectral subspaces,  see Corollaries \ref{wniosek o spektralnych podprzestrzeniach} and \ref{spectral-sp}. We write $\M_{k}\star\mathcal{M}_{l}$ for the linear span of elements $a \star b$, $a\in \M_k$, $b\in \M_l$.
\begin{prop}
\label{**}
The spaces $\M_k$ define  a $\Z$-graded algebra structure on $\M(\A)$. Namely
$$
\M(\A)=\bigoplus_{k\in \Z} \M_k,$$
 and for every $k$, $l\in \Z$ we have the following relations
$$
\M_{k}^*=\M_{-k},\qquad\M_{k}\star\mathcal{M}_{l}\subset \M_{k+l}.
$$
In particular, $\M_{0}$ is a $^*$-algebra,  $\M_{k}\star \M_{-k}$
is a self-adjoint two sided  ideal in $\M_{0}$.
\end{prop}
\label{coef-alg}
\begin{Proof}
Relations  $\M_k^*=\M_{-k}$, $\M_k\star \M_l\subset \M_{k+l}$ and
$(\M_k \star \M_k^*)^*=(\M_k \star \M_k^*)$ can be checked by
means of an elementary matrix calculus. Using these relations we
get
$$
\M_0 \star (\M_k \star \M_k^*)= (\M_0 \star \M_k)\star \M_k^*\subset \M_k\star \M_k^* ,
$$
$$
(\M_k\star \M_k^*)\star \M_0 = \M_k\star (\M_k^*\star
\M_0)\subset  (\M_k \star \M_k^*),
$$
and thus $\M_k\star \M_k^*$ is an ideal in $\M_0$.
\end{Proof}

Proposition \ref{**} indicates in particular that $\M_0$ may be  regarded as  a coefficient algebra
 in the sense of \cite{Leb-Odz} for  $\M(\A)$.
This will be shown explicitly in Proposition~\ref{nie dam rady umre}.

\begin{cor}\label{wniosek o spektralnych podprzestrzeniach}
  Let $(\pi,U,H)$ be  a  representation of $(\A,\alpha)$ and let
  $
B_k= \Psi_{(\pi, U)}(\M_k)
 $
 be the linear space consisting of the elements of the form
 $$
\sum_{n=0}^N U^{*n} \pi(a_{n}^{(k)}) U^{n+k}, \quad  \textrm{ if }\,\, k \geq 0, \quad \textrm{ or}
\quad \sum_{n=0}^N U^{*n+|k|} \pi(a_{n}^{(k)})U^{n}, \quad  \textrm{ if }\,\, k < 0.
$$
Then for every $k$ and $l\in \Z$ we have the following relations
$$
B_k^*=B_{-k},\qquad B_kB_l\subset B_{k+l}
$$
In particular, $B_0$ is a $C^*$-algebra,  $B_k B_k^*$ is a self-adjoint two sided  ideal in $B_0$.
\end{cor}
The importance of $B_0$ was observed in \cite{Leb-Odz} and is clarified  by the next proposition,
see \cite[Proposition 2.4]{Leb-Odz}.
\begin{prop}
Let $(\pi,U,H)$ be  a representation of $(\A,\alpha)$ and adopt
the notation from Proposition  \ref{dense subalgera structure prop} and
Corollary \ref{wniosek o spektralnych podprzestrzeniach}. Every
element $a \in C_0^*(\pi(\A),U)$ can be presented in the form
$$
a= \sum_{k=1}^{\infty} U^{*k} a_{-k} + \sum_{k=0}^{\infty}  a_{k}U^{*k}
$$
where $a_{-k} \in B_0\pi(\alpha^k(1))$,  $a_{k} \in \pi(\alpha^k(1))  B_0$, $k\in \N$,
and only finite number of these coefficients  are non-zero.
\end{prop}
We will now formulate a similar result concerning $\M(\A)$. For
each $k\in \Z$ we define a mapping $\NN_k:\M(\A)\to \M_0$, $k\in\Z$,
that carries the $k$-diagonal onto a $0$-diagonal and delete all the
remaining ones. Namely, for $a = \{ a_{n}^{(k)}  \}$ we set
$$
\left[\NN_k (a)\right]_{n}^{(m)}=
\begin{cases}
a_{n}^{(k)} & \textrm{ if }  m=0,\\
0 & \textrm{ otherwise },
\end{cases} \qquad \quad k \in \Z.
$$
One readily checks that   for $k\geq 0$ we have $\NN_k(\M_k)
=\M_0\star \alpha^k(1)$, $\NN_{-k}(\M_{-k})= \alpha^k(1) \star
\M_0$. Thus    the algebra  $\M_0$ consists of elements
that play the role of Fourier coefficients in $\M(\A)$.
\begin{prop}\label{nie dam rady umre}
Every element $a$ of  $\M(\A)$ is uniquely presented in the form
$$
a= \sum_{k=1}^{\infty} u^{*k}\star a_{-k} + \sum_{k=0}^{\infty}  a_{k}\star  u^{*k}
$$
where $u$ is given by (\ref{notation stupid}) and
 $a_{-k} \in \M_0\star \alpha^k(1)$,  $a_{k}
\in \alpha^k(1) \star \M_0$, $k\in \N$, and only finite number of
these coefficients  is non-zero. Namely, $a_k=\NN_k(a)$ for $k\in
\Z$.
\end{prop}

\subsection{Norm evaluation of elements in $C_0^*(\pi(\A),U)$}\label{norm}

In this subsection we gather a number of technical results
concerning norm evaluation of elements in $C_0^*(\pi(\A),U)$. We
will make use of these results in the sequel.
\smallskip

The mappings $\NN_k:\M_k \to \M_0$ factor through $\Phi_{(\pi,U)}$ to the mappings \mbox{$N_k:B_k\to B_0$.}
\begin{prop}\label{proposition for k-diagonals}
Let $(\pi,U,H)$ be  a  representation of $(\A,\alpha)$ and let
$k\in \Z$.  Then the norm $\|a\|$ of an element $a\in B_k$
corresponding to the matrix $\{a_{n}^{(m)}\}_{n\in \N, m\in \Z}$
in $\M_k$ is given by
$$
\lim_{n\to \infty} \max \left\{\max_{i=1,...,n}\Big \|(1-U^*U)\sum_{j=0}^{i}
\pi(\alpha^{i-j}(a_{j}^{(k)}))\Big\|,\, \Big\|U^*U \pi(a_{n}^{(k)})\Big\| \right\}.
$$
In particular, the mapping $N_k:B_k\to B_0$ given by
$$
N_k (\Phi_{(\pi,U)}(a))=\Phi_{(\pi,U)}(\NN_k (a)), \qquad a\in \M_k,
$$
is a well defined linear isometry  establishing the following isometric isomorphisms
$$
B_k\cong B_0\alpha^k(1),\quad \textrm{ if }\,\, k\geq 0,\qquad B_{k}\cong \alpha^{|k|}(1) B_0,\quad
\textrm{ if }\,\, k < 0.
$$
\end{prop}
\begin{Proof} Let us assume that $k \geq 0$.
Let $N$ be such that $a_{m}^{(k)}=0$ for $m>N$, that is
$$
a= \sum_{m=0}^{N} U^{*m}\pi(a_{m}^{(k)})U^{m+k}.
$$ Then, similarly  as it was done in the proof of \cite[Prop. 3.1]{kwa4},
one sees that defining
$$
a_i=(1-U^*U) \pi\left(\sum_{j=0}^{i} \alpha^{i-j}(a_{j}^{(k)})\right), \quad  i=0,...,N,\qquad a_{N+1}= U^*U \pi(a_{N}^{(k)}),
$$
we have
 $$
a = \big(a_0 +U^*a_1U +... + U^{*N}(a_N + a_{N+1})U^N\big) U^k
$$
and
$$
a_i\in (1-U^*U)\pi(\alpha^i(1)\A \alpha^{i+k}(1)), \quad  i=0,...,N,
\quad a_{N+1} \in U^*U\pi(\alpha^N(1)\A
\alpha^{N+k}(1)) ,
$$
Hence it follows that
$$
U^{*i} a_i U^{i} \in (U^{*i}U^{i}-U^{*i+1}U^{i+1})\pi(A)U^kU^{*k}, \quad  i=0,...,N,
 $$
 and
 $$
  U^{*N}a_{N+1}U^N \in U^{*N+1}U^{N+1}\pi(\A)U^kU^{*k}
$$
These relations and the fact that $U^{*i}U^{i}-U^{*i+1}U^{i+1}$, $i=0,..,N$, and $U^{*N+1}U^{N+1}$ are
pairwise orthogonal projections lying in $\pi(\A)'$ (cf. Lemma \ref{iteration of representations} or better \cite[Prop 3.6]{Leb-Odz}),
imply the following equalities
 \begin{align*}
 \|a\|&=\| (a_0 +U^*a_1U +... + U^{*N}(a_N+ a_{N+1}) U^N) U^kU^{*k}\|
\\
& =\|a_0 +U^*a_1U +... + U^{*N}(a_N+ a_{N+1})U^N\|
\\
&=\max \{\max_{i=0,...,N}\|U^{*i}a_iU^i\|, \|U^{*N}(a_{N+1})U^N\|\}
\\
&=\max \{\max_{i=0,...,N}\|a_i\|, \|a_{N+1}\|\}
\end{align*}
where the final equality follows from that the linear mapping
$ a \to U^{*n} a U^{n}$  is isometric on $\pi(\al^{n}(1)A\al^{n}(1))=U^nU^{*n} \pi(A)U^nU^{*n}$, $n\in \N$. Since $N$ was arbitrary (sufficiently large) this proves the assertion in the case when $k \geq 0$.

In the case of negative $k$ one may apply the part of proposition  proved above  to the adjoint $a^*$
of the element $a$ and thus obtain the hypotheses.
\end{Proof}
We denote by
$$
d(a, K)=\inf_{b\in K} \|a-b\|
$$
the usual distance of an element $a$ from the set $K$. The
definition of an ideal  of covariance  (Definition \ref{kowariant
rep defn 2})  and a known  fact expressing quotient norms in terms
of projections, see for instance, \cite[Lemma 10.1.6]{Kadison}
gives us the following
\begin{cor}\label{corollary with distance}
If $(\pi,U,H)$ is a faithful $J$-covariant representation of
$(\A,\alpha)$  and $I$ denotes the kernel of $\alpha$, then the
norm of an element $a\in B_k$ corresponding to a matrix
$\{a_{n}^{(m)}\}_{n\in \N, m\in \Z}$  in $\M_k$ is given by
$$
\|a\|=\lim_{n\to \infty} \max \left\{\max_{i=1,...,n}\big\{ d\big(\sum_{j=0}^{i}
\alpha^{i-j}(a_{j}^{(k)}),J\big)\big\},\, d(a_{n}^{(k)},I) \right\}.
$$
In particular, if $J$ is a fixed ideal orthogonal to $I$ and
$(\pi,U,H)$ is a faithful $J$-covariant representation, then the
spaces $B_k$, $k\in\Z$ do not depend on its choice.
\end{cor}
We showed in Proposition \ref{proposition for k-diagonals} that
for an arbitrary  representation $(\pi,U,H)$ of $(A,\alpha)$, the
mappings $\NN_k$ factor through to the mappings $N_k$ acting on
spaces $B_k$. In general, however,
 $\NN_k:\M(\A)\to B_0$  do not factor through $\Psi_{(\pi,U)}$ to the mappings acting on
 the algebra $C_0^*(\pi(\A),U)$. In fact, this is the case if and only if the  representation  $(\pi,U,H)$
 satisfies a  certain property we are just about to introduce.
\begin{defn}
\label{*} We will say that a faithful  representation $(\pi,U,H)$
of $(\A,\alpha)$ possesses  {\em property} $(*)$ if for any $a\in
C_0^*(\pi(\A),U)$ given by a matrix
$\{a_{mn}\}_{m,n\in\N}\in\M(\A)$  the inequality
$$
\|\sum_{m\in\N} U^{*m}\pi(a_{m,m})U^m \| \leq \|\sum_{m,n\in\N} U^{*m}\pi(a_{m,n})U^n \|,  \qquad\qquad (*)
$$
holds. In view of Corollary \ref{corollary with distance}  the above equality could be equivalently stated
in the form    $$
\lim_{n\to \infty} \max \left\{\max_{i=1,...,n}\big\{ d\big(\sum_{j=0}^{i}
\alpha^{i-j}(a_{j,j}),J\big)\big\},\, d(a_{n,n},I) \right\} \leq \|a\|,\quad (*)
$$
where  $I$ is the kernel of $\alpha$ and $J$ is the ideal of
covariance of  $(\pi,U,H)$.
\end{defn}
The next result, which follow immediately from \cite[Thm. 2.8]{Leb-Odz}, indicates that
under the fulfillment of property~(*) elements of $B_0$ play the
role of 'Fourier' coefficients in the algebra $C_0^*(\pi(\A),U)$.
\begin{thm}
Let $(\pi,U,H)$ be a faithful  representation of $(A,\alpha)$
possessing property $(*)$ then the mappings $N_k:
C_0^*(\pi(\A),U)\to B_0$, $k\in \Z$,  given by formulae
\begin{equation}\label{rzuty na coefficienty}
 N_k( \Phi_{(\pi,U)}(a)) = \sum_{n\in\N} U^{*n}\pi(a_{n}^{(k)})U^{n},
\end{equation}
where $\{a_{n}^{(m)}\}_{n\in \N, m\in \Z}\in \M(\A)$, are well defined contractions and thus
they  extend uniquely to bounded operators on $C^*(\pi(\A),U)$. In particular,
every element $a \in C_0^*(\pi(\A),U)$ can  be uniquely presented in the form
$$
a= \sum_{k=1}^{\infty} U^{*k} a_{-k} + \sum_{k=0}^{\infty}  a_{k}U^{*k}
$$
where $a_{-k} \in B_0\pi(\alpha^k(1))$,  $a_{k} \in \pi(\alpha^k(1))  B_0$, $k\in \N$,
namely, $a_k=N_k(a)$, $k\in \Z$.
\end{thm}

Let  us also recall  \cite[Thm. 2.11]{Leb-Odz}.
\begin{thm}\label{3a.N}
If $(\pi,U,H)$ possesses   property $(*)$, then for any element $a$ in $C_0^*(\pi(\A),U)$  we have
\begin{equation}\label{be3.131}
\Vert a \Vert = \lim_{k\to\infty}
\sqrt[\leftroot{-2}\uproot{1}\scriptstyle 4k]{
\left\Vert N_0 \left[ (aa^*)^{2k}\right]\right\Vert }
\end{equation}
where $N_0$ is the mapping defined by \eqref{rzuty na coefficienty}.
\end{thm}

Using the above results  one sees that in the presence of property $(*)$ the  norm of an
element $a\in C_0^*(\pi(\A),U)$ may be calculated only in terms of the elements
of $\A$. Indeed, as $N_0 \left[ (aa^*)^{2k}\right]$ belongs to $B_0$ one can apply Corollary
\ref{corollary with distance} to calculate $\| N_0 \left[ (aa^*)^{2k}\right]\|$ in terms of
the matrix from $\M(\A)$ corresponding to $a$. However in practice,   calculation of
the matrix corresponding to the element $(aa^*)^{2k}$ starting from $a$, see formula \eqref{star1},
seems to be an extremely difficult task.

\subsection{Crossed product defined by matrix calculus}
\label{cr-matr}

The foregoing observations make it now possible  to give one more  'internal' definition of the crossed product. This is the aim of the present subsection.
\smallskip

The set $\M(\A)$ with operations \eqref{add1},
\eqref{mulscal1}, \eqref{invol1}, \eqref{star1} is an algebra with involution. We define
a seminorm  on $\M(\A)$ that will depend on the choice of an orthogonal ideal. Let $J$ be
a fixed ideal in $\A$ having zero intersection with the kernel of $\alpha$.
Let
$$
    \|| a\||_J:=\sum_{k\in \Z} \lim_{n\to \infty} \max \left\{\max_{i=1,...,n}\big\{ d\big(\sum_{j=0}^{i}
    \alpha^{i-j}(a_{j}^{(k)}),J\big)\big\},\, d(a_{n}^{(k)},I) \right\}
    $$
    where $a=\{a_{n}^{(k)}\}_{n\in \N,k\in \Z}\in \M(\A)$.
    \begin{prop}
The  function  $\|| \cdot \||_J$ defined above is a  seminorm on $\M(\A)$ which is  $^*$-invariant and submulitplicative.
\end{prop}
\begin{Proof}
Let $(\pi,U,H)$ be a faithful  representation of $(\A,\alpha)$ and
 $J$ be its covariance ideal. Such representation does exist by Theorem \ref{existance
theorem}. Then in view of  Corollary \ref{corollary with distance}
for every $a\in \M_k$, $k\in \Z$, we have
$$
    \|| a\||= \| \Phi_{(\pi,U)}(a)\|
    $$
where $\Phi_{(\pi,U)}:\M(\A)\to L(H)$ is the $^*$-homomorphism defined in
Proposition~\ref{dense subalgera structure prop}. Thus since every element $a\in \M(\A)$ can be presented in the
form $a= \sum_{k\in \Z} a^{(k)}$ where  $a^{(k)}\in \M_k$ one easily sees
that $\||| \cdot \|||$ is $^*$-invariant seminorm. To show that it is
submultiplicative take $a= \sum_{k\in \Z} a^{(k)}\in \M(\A)$ and  $b= \sum_{k\in \Z} b^{(k)}\in \M(\A)$
such that $a^{(k)}, b^{(k)}\in \M_k$.  Then
$$
\|| a \star b \|| = \||\sum_{k\in \Z} a^{(k)}\star \sum_{l\in \Z} b^{(l)}\||=\||\sum_{k\in \Z}
\sum_{l\in \Z} a^{(k)}\star b^{(l)}\||\leq \sum_{k,l\in \Z}  \||a^{(k)}\star b^{(l)}\||
$$
$$
= \sum_{k,l\in \Z}  \|\Phi_{(\pi,U)}(a^{(k)}\star b^{(l)})\| \leq  \sum_{k,l\in \Z}  \|\Phi_{(\pi,U)}(a^{(k)})\|
\cdot  \| \Phi_{(\pi,U)}(b^{(l)})\|
$$
$$
 = \sum_{k,l\in \Z}  \||a^{(k)} \|| \cdot \||b^{(l)}\||
=\sum_{k \in\Z} \||a^{(k)}\|| \sum_{l\in \Z} \||b^{(l)}\||  =\||a\|| \cdot\|| b\||.
$$
\end{Proof}

\begin{defn}\label{cr-pr-def}
Let $(\A,\alpha)$ be a $C^*$-dynamical system and $J$ an ideal in $\A$ having  zero intersection with the kernel
of $\alpha$. The \emph{crossed product} $
C^*(\A,\alpha,J)
$ of $\A$ by $\alpha$ associated with $J$ is   the enveloping $C^*$-algebra  of the quotient $^*$-algebra $\M(\A)/ \|| \cdot \||_J$.
\end{defn}
 Regardless of $J$, composing the quotient map with natural embedding of $\A$ into $\M(\A)$
 one has an embedding of $\A$ into $C^*(\A,\alpha,J)$. Moreover, denoting by $\hat{u}$
 an element of $C^*(\A,\alpha,J)$ corresponding to $u\in \M(\A)$ (see (\ref{notation stupid})),
 one sees that $C^*(\A,\alpha,J)$ is generated by $\A$ and $\hat{u}$.

 The equivalence of this definition and that introduced previously (Definition~\ref{crossed product defn})  will be established in the next section (Proposition~\ref{crossed equiv}).

 \section{Isomorphism theorems and faithful representations}\label{isomorph}

 Once a universal object (the crossed product) is defined it is reasonable to have its faithful representation.
 This section is devoted to the description of the properties of such representations  and, in particular, we establish the equivalence of two previously mentioned definitions of the crossed product. In addition we present one more alternative crossed product construction based on \cite{Ant-Bakht-Leb} approach and prove faithfulness by means of the topologically free action on the arising coefficient algebras.

\subsection{Isomorphism Theorem}\label{isomor-theor}

\begin{thm}[\textbf{Isomorphism Theorem}]\label{isomorphiasm theorem}
\label{iso} Let $J$ be an ideal in $\A$ having zero intersection
with the kernel $I$ of\, $\alpha$ and let $(\pi_i,U_i,H_i)$,
$i=1,2$, be  faithful $J$-covariant representations of
$(\A,\alpha)$  possessing  property $(*)$. Then the relations
$$ \Phi(\pi_1(a)):=\pi_2(a),\quad a\in \A,\qquad
\Phi(U_1):=U_2
$$
gives rise to an isomorphism between the $C^*$-algebras $C^*(\pi_1(\A),U_1)$
and $C^*(\pi_2(\A),U_2)$.
\end{thm}
\begin{Proof} Let $B_{0,i}$ be a $^*$-algebra consisting of elements of the form
$\sum_{n=0}^N U^{*n}_i\pi_i(a_n)U_i^n$, $i=1,2$.
In view of Corollary \ref{corollary with distance}, $\Phi$ extends to an isometric isomorphism
from $B_{0,1}$ onto $B_{0,2}$. Moreover, we have
$$
\Phi(U_1aU_1^*)=U_2(\Phi(a))U^*_2, \qquad a \in B_{0,1}.
$$
Hence the assumptions of \cite[Theorem 2.13]{Leb-Odz} are satisfied and the hypotheses follows.
\end{Proof}
\begin{cor}\label{spectral-sp}
If $(\pi,U,H)$ possesses property $(*)$, then we have a point-wise continous action
$\gamma$ of the group $S^1$ on $C^*(\pi(\A),U)$ by authomorphisms
given by
$$
\gamma_z(\pi(a)):=\pi(a), \quad a\in \A, \qquad  \gamma_z(U) := z U, \qquad z\in S^1.
$$
Moreover the spaces $\overline{B}_k$ are the spectral subspaces corresponding to this action, that is we have
$$
\overline{B}_k=\{ a \in C^*(\A,U): \gamma_z(a)= z^k a\}.
$$
In particular, the $C^*$-algebra $\overline{B}_0$ is the fixed point algebra for $\gamma$.
\end{cor}
\begin{Proof} Let $(\pi,U,H)$ be a $J$-covariant representation of $(A,\alpha)$ possessing property $(*)$ and let $z\in S^1$. It is clear that $(\pi,zU,H)$
is also a $J$-covariant representation of $(\A,\alpha)$  and
$(\pi,zU,H)$ possesses property $(*)$. Hence by virtue of Theorem
\ref{isomorphiasm theorem}, $\gamma_z$  extends to an isomorphism of
$C^*(\pi(\A),U)=C^*(\pi(\A),zU)$. The remaining part of the
statement is obvious.
\end{Proof}
\smallskip

The next  theorem is an immediate corollary of the previous statements and \cite[Thm. 2.15]{Leb-Odz}.
It is another manifestation of the fact that  the elements $N_k (a)$, $k\in\Z$, should be considered as
Fourier coefficients for $a \in C^*(\pi(\A),U)$.
\begin{thm}
\label{uniqueNk}
Let $(\pi,U,H)$ possess property $(*)$ and let
\begin{center}
$a \in C^*(\pi(\A),U)$.
\end{center}
Then the following conditions are equivalent:

\smallskip
\quad\llap{$(i)$}\ \ $a=0;$

\smallskip
\quad\llap{$(ii)$}\ \ $N_k (a)=0$, \,$k\in\Z;$

\smallskip
\quad\llap{$(iii)$}\ \ $N_0 (a^*a)=0$.

\end{thm}

The results presented above give us a possibility to write out a
  criterion for a representation of the crossed product to be faithful.

\begin{thm}\label{5.4} Let $C^*(\A,\alpha, J)$ be the crossed product given by
Definition~\ref{cr-pr-def} and   $(\pi,U,H)$ be a faithful
$J$-covariant representations of $(\A,\alpha)$. Then the relations
\begin{equation}\label{e-fath}
(\pi\times U)(a)=\pi(a), \ a\in \A;\qquad (\pi\times U)(\hat{u})=U
\end{equation}
determine in a unique way an epimorphism $\pi\times U: C^*(\A,\alpha, J) \to C^*(\pi(\A),U)$. Moreover
$\pi\times U$ is an isomorphism iff $(\pi,U,H)$ possesses property $(^*)$.
\end{thm}

\subsection{Construction of faithful representations}\label{faith-covar}
 Given a faithful $J$-covariant representation  of $(\A,\alpha)$   one can construct a covariant representation of $(\A,\alpha)$ possessing property~$(^*)$ thus obtaining a faithful representation of $ C^*(\A,\alpha,J)$ (in view of Theorem~\ref{5.4}). Actually, we exploit  the standard argument cf.~\cite{Ant-Bakht-Leb}.

\begin{prop}
For any faithful $J$-covariant representation $(\widetilde{\pi},\widetilde{U},\widetilde{H})$ of $(\A,\alpha)$ the triple $(\pi,U, { H})$ where $H :=l^2 ({\mathbb Z}, \widetilde{H})$,
\begin{equation}
(\pi (a)\xi )_n := \widetilde{\pi} (a) (\xi_n), \quad \textrm{ and }\quad
(U\xi )_n := \widetilde{U} (\xi_{n-1})\label{18e}
\end{equation}
for $a\in \A$, $ \xi = \{ \xi_n  \}_{n\in\Z}\in H=l^2 ({\mathbb Z}, \widetilde{H})$, is a faithful $J$-covariant representation which integrate via  \eqref{e-fath} to a faithful representation $(\pi\times U)$ of $ C^*(\A,\alpha,J)$.
\end{prop}
\begin{Proof}
Routine verification shows that  $(\pi,U, { H})$  is a faithful
$J$-covariant representation of $(\A,\alpha)$. By
Theorem~\ref{5.4} it suffices to verify that $(\pi,U, { H})$
possesses property $(^*)$. To this end take any $N\in \N$,
$a_{m,n}\in \A$, $n,m=0,1...,N$, and note that by the explicit
form of  \eqref{18e} we have
\begin{equation}\label{eq}
\Vert\sum_{m=0}^N {\widetilde{U}}^{*m}\widetilde{\pi}(a_{m,m})\widetilde{U}^m
\Vert = \Vert \sum_{m=0}^N U^{*m}\pi(a_{m,n})U^m \Vert
\end{equation}
For a given $\varepsilon > 0$ there exists   a vector
$\eta \in \widetilde{H}$ such that $\Vert \eta \Vert =1$ and
\begin{equation}\label{e*}
\Vert \left(\sum_{m=0}^N {\widetilde{U}}^{*m}\widetilde{\pi}(a_{m,m})\widetilde{U}^m\right)
\eta \Vert > \Vert\sum_{m=0}^N {\widetilde{U}}^{*m}\widetilde{\pi}(a_{m,m})\widetilde{U}^m
\Vert - \varepsilon .
\end{equation}
Set $\xi = \{ \xi_n  \}_{n\in
\Z} \in l^2 ({\mathbb Z}, \widetilde{H})$ by $\xi_n = \delta_{0,n}\eta $,
where $\delta_{i,j}$ is the Kronecker symbol. We have that $\Vert
\xi \Vert = 1$ and  \eqref{18e} along with  (\ref{eq}) and  (\ref{e*}) imply
\begin{equation}\label{e21}
\Vert \left(\sum_{m=0}^N U^{*m}\pi(a_{m,n})U^m\right)\xi  \Vert>
\Vert \sum_{m=0}^N U^{*m}\pi(a_{m,n})U^m \Vert
-\varepsilon
\end{equation}
Now the explicit form of \ $\left(\sum_{m,n=0}^N U^{*m}\pi(a_{m,n})U^n\right)
\xi $    \ and
\eqref{e21} imply
\[
\Vert \sum_{m,n=0}^N U^{*m}\pi(a_{m,n})U^n
       \Vert^2 \ge
\Vert  \left(\sum_{m,n=0}^N U^{*m}\pi(a_{m,n})U^n\right)\xi
\Vert^2\ge
\]
\begin{equation}\label{end}
\ge    \Vert \left(\sum_{m=0}^N U^{*m}\pi(a_{m,n})U^m\right)\xi
\Vert^2     > \left(\Vert\sum_{m=0}^N U^{*m}\pi(a_{m,m})U^m\Vert -
\varepsilon\right)^2
\end{equation}
 which by  the
arbitrariness of $\varepsilon$ proves property $(^*)$ for
$(\pi,U,{ H})$. \end{Proof}

Following the foregoing construction one can also  arrive at the next
 \begin{prop}\label{star property*}
 The crossed product  $ C^*(\A,\alpha,J)$, given by Definition~\ref{crossed product defn}  possesses property $(*)$, that is the algebra $\A$ is embedded in the crossed product  $ C^*(\A,\alpha,J)$ and   for any  $N\in \N$ and  $a_{m,n}\in \A$, $n,m=0,1...,N$ the following inequality holds
  \begin{equation}\label{*-cross}
\|\sum_{m=0}^N u^{*m}a_{m,m}u^m \| \leq \|\sum_{m,n=0}^N u^{*m}a_{m,n}u^n \|.
\end{equation}
\end{prop}
 \begin{Proof}
 Let  $ C^*(\A,\alpha,J)$ be given by Definition~\ref{crossed product defn}.  Take any faithful
 representation $\overline{\pi}: C^*(\A,\alpha,J) \to L(\widetilde{H})$ of $ C^*(\A,\alpha,J)$ in a Hilbert space $\widetilde{H}$
  and 'disintegrate'  $\overline{\pi}$ to $(\widetilde{\pi},\widetilde{U},\widetilde{H})$, i.e.  let $\widetilde{\pi}:=\overline{\pi}|_\A$
  and $\widetilde{U}:=\overline{\pi}(u)$. Then $(\widetilde{\pi},\widetilde{U},\widetilde{H})$  is  a  faithful
  $J$-covariant representation  of $(\A,\alpha)$   (note that Theorem~\ref{existance theorem} and Definition~\ref{crossed product defn} imply that $\widetilde{\pi}$ is a faithful representation of $\A$). Consider the space ${ H} =l^2 ({\mathbb Z}, \widetilde{H})$ and
representation $(\pi,U,H)$  given by   \eqref{18e}. By the
universality of  $C^*(\A,\alpha,J)$ (Definition~\ref{crossed
product defn}) this representation give rise to a representation
of $C^*(\A,\alpha,J)$. And therefore for any any $N\in \N$ and
$a_{m,n}\in \A$, $n,m=0,1...,N$ one has
$$
\|\sum_{m,n=0}^N u^{*m}a_{m,n}u^n \|\ge \Vert \sum_{m,n=0}^N U^{*m}\pi(a_{m,n})U^n
       \Vert.
$$
Moreover  by the explicit form of  \eqref{18e} we have
$$
\|\sum_{m=0}^N u^{*m}a_{m,m}u^m \|= \Vert \sum_{m=0}^N U^{*m}\pi(a_{m,m})U^m
       \Vert.
$$
Now \eqref{*-cross}  follows from \eqref{end}.
\end{Proof}

The above results imply
\begin{prop}\label{crossed equiv}
The crossed products given by Definitions~\ref{crossed product defn} and~\ref{cr-pr-def} are canonically
isomorphic.
\end{prop}
\begin{Proof} Let  $C^*(\A,\alpha,J)$ be the crossed product given by Definition~\ref{crossed product defn}.
Set
$$
\pi_1:\A\to\hat{\A}, \ \pi_1 (a):= \hat{a}, \qquad \textrm{ and }\qquad
U_1:=\hat{u},
$$
where $\hat{a}$ and $\hat{u}$ are the universal elements, corresponding to $a\in\A$ and $u$ according to Definition~\ref{crossed product defn}. Then
$$
C^*(\A,\alpha,J)=C^*(\pi_1(\A),U_1).
$$
By Definition~\ref{crossed product defn} and Theorem \ref{existance theorem},  $\pi_1$ establishes an isomorphism between $\A$ and $\hat{\A}$. By Proposition~\ref{star property*} \ $C^*(\hat{\A},\hat{u})$ possesses
property~$(^*)$.

Now, let $C^*(\A,\alpha,J)$ be the crossed product given by
Definition~\ref{cr-pr-def}, and let $(\pi,U,H)$ be a faithful
$J$-covariant representation of $(\A,\alpha)$ possessing property
$(^*)$ (such representation does exist by
Remark~\ref{faith-covar}). Set $(\pi_2,U_2, H)$ by
\begin{equation}\label{e-fath*}
\pi_2(a):=(\pi\times U)(a)=\pi(a), \ a\in \A;\qquad U_2:=(\pi\times U)(\hat{u})=U
\end{equation}
(cf. \eqref{e-fath}). Then by Theorem~\ref{5.4}
$$
C^*(\pi_2(\A),U_2)\cong C^*(\A,\alpha,J)
$$
and $(\pi_2,U_2, H)$ possesses property $(^*)$. Thus by Theorem~\ref{isomorphiasm theorem}
$$
C^*(\pi_1(\A),U_1)\cong C^*(\pi_2(\A),U_2).
$$
\end{Proof}
\subsection{Topological freeness}\label{ABL=Kwa-Leb} Crossed product $C^*(\A,\alpha,J)$ can also be constructed by means of the crossed product introduced in \cite{Ant-Bakht-Leb}. Here we present this  construction. It  can  be considered as one more alternative definition of $C^*(\A,\alpha,J)$.

Let $(\A,\alpha)$  be a $C^*$-dynamical system,   $J$ be an ideal orthogonal to the kernel of $\alpha$, and
 $(\pi,U,H)$ be a faithful $J$-covariant representation  of $(\A,\alpha)$. Consider the algebra
\begin{equation}\label{ext}
\B:=C^*\left(\bigcup_{n\in \N}U^{*n}\pi(\A)U^{n}\right).
\end{equation}
By Corollary~\ref{corollary with distance} algebra $\B$ can be
described in terms of  $(\A,  \alpha)$ and $J$,  and  therefore it
does not depend on the choice of a faithful $J$-covariant
representation $(\pi,U,H)$. Routine calculation (cf.
\cite[Prop. 3.10.]{Leb-Odz}) shows that
\begin{equation}\label{B-2}
U\B U^*\subset \B, \qquad U^*\B U\subset \B, \qquad  U^*U\in Z(\B).
\end{equation}
Thus $\B$ is a \emph{coefficient algebra}  in the sense of \cite{Leb-Odz}. In particular
$$
\alpha:\B\to\B, \qquad \alpha (\cdot):= U(\cdot)U^*
$$
is an endomorphism of $\B$ (identifying $\pi(A)$ with $A$ the mapping $U(\cdot)U^*$ extends the endomorphism $\al:A\to A$  and  the notational collision disappears) and
$$
\mathcal{L}:\B\to\B, \qquad \mathcal{L} (\cdot):=U^*(\cdot)U
$$
 is a \emph{complete transfer operator} in the sense of  \cite{Bakht-Leb}. Note also  that the mapping $\mathcal{L}$ is
 defined uniquely by the $C^*$-dynamical system $(\B,\alpha)$ (\cite[Theorem 2.8.]{Bakht-Leb}), and therefore it is uniquely
 defined by $(\A,\alpha)$ and $J$.
In particular,   Proposition~\ref{star property*} and \cite[Theorem 3.5.]{Ant-Bakht-Leb} imply the following
\begin{prop}\label{crossed-ABL1} We have a natural  isomorphism
\begin{equation}
\label{crossed-ABL}
C^*(\A,\alpha,J)\cong\B\times_\alpha\Z,
\end{equation}
where on the right hand side stands the crossed product of \cite[Def. 2.6.]{Ant-Bakht-Leb}.
\end{prop}

\begin{rem}
Thus the alternative construction of crossed product involve essentially two steps: 1)~extending an
\emph{irreversible} system  $(\A,\alpha)$  up to a \emph{reversible} system on $(\B,\alpha)$ (here $\mathcal{L}$ plays the role of the inverse to $\alpha$),
and then 2)~attaching the crossed product from \cite{Ant-Bakht-Leb} to the extended system.

We have to stress that the general procedure of extension of $(\A,\alpha)$ up to $(\B,\alpha)$ by means of \eqref{ext} involves the ideal $J$ but 'does not see' it explicitely and namely the material of the present paper shows that by means of these orthogonal ideals
all the possible extensions are parametrised (recall in this connection the discussion in \cite[Section 5]{Ant-Bakht-Leb}).
\end{rem}

The  isomorphism \eqref{crossed-ABL} and the results of
\cite{kwa-ck} give us a possibility to obtain one more isomorphism
theorem which can be written in terms of topological freeness of
the action $\alpha$ on $\B$. Indeed, it follows immediately from
\eqref{B-2} that
$$
\mathcal{L}:\alpha(\B)\to \mathcal{L}(\B)\,\,\,\,\textrm{ and }\,\,\,\, \alpha:\mathcal{L}(\B)\to \alpha(\B)
$$
are mutually inverse isomorphisms $\mathcal{L}(\B)=\mathcal{L}(1)\B$ is an ideal in $\B$
 and $\alpha(\B)=\alpha(1)\B\alpha(\B)$ is a hereditary subalgebra of $\B$.  Therefore, cf.  e.g. \cite[Thm 5.5.5]{Murphy},  we may naturally identify the spectra $\widehat{ \alpha(\B)}$ and $\widehat{\mathcal{L}(\B)}$  of $\alpha(\B)$ and $\mathcal{L}(\B)$ with open subsets of the spectrum $\widehat{\B}$ of $\B$, and then
\begin{equation}\label{identifications of spectra}
\widehat{ \alpha(\B)}=\{\pi \in  \widehat{\B}: \pi(\alpha(1))\neq 0\},\qquad \widehat{ \mathcal{L}(\B)}=\{\pi \in  \widehat{\B}: \pi(\mathcal{L}(1))\neq 0\}.
\end{equation}
Under the above  identifications   the dual $\widehat{\alpha}:\widehat{ \alpha(\B)} \to \widehat{ \mathcal{L}(\B)}$ to the isomorphism $\alpha:\mathcal{L}(\B)\to \alpha(\B)$ becomes a partial homeomorphism of $\B$.  More precisely, let $\pi:\B\to L(H)$  be an irreducible representation. If   $\pi(\alpha(1))\neq 0$, then
\begin{equation}\label{homeomorphism dual to delta}
\widehat{\alpha}(\pi)=\pi\circ \alpha: \B\to L(\alpha(1)H)
\end{equation}
is an irreducible representation such that $\widehat{\alpha}(\pi)(\L(1))\neq 0$. Conversely if $\pi(\mathcal{L}(1))\neq 0$, then  $\widehat{\alpha}^{-1}(\pi)=\widehat{\mathcal{L}}(\pi)$  is a unique (up to unitary equivalence) irreducible extension of the representation
$
\pi\circ \mathcal{L}: \alpha(1)\B\alpha(1)\to L(H).
$
\begin{defn}
We say that  $\widehat{\alpha}$ where $\widehat{\alpha}:\widehat{ \alpha(\B)} \to \widehat{ \mathcal{L}(\B)}$ is a homeomorphism  given by \eqref{homeomorphism dual to delta}, between the open subsets \eqref{identifications of spectra} of $\widehat{\B}$, is a \emph{partial homeomorphism dual to the endomorphism }$\alpha:\B\to \B$. We call the pair $(\widehat{\B}, \widehat{\alpha})$ the \emph{partial dynamical system dual to the $C^*$-dynamical system} $(\B,\alpha)$.
\end{defn}
We recall  that  a partial homeomorphism  of a topological space, i.e. a homeomorphism between open subsets,   is   {\em
topologically free} if  for any $n\in \N$ the set of periodic points of period $n$ has empty interior.
Applying \cite[Thm. 2.24]{kwa-ck} we arrive at the following result
\begin{thm}[\textbf{Isomorphism theorem and topologically free action}]\label{topolo disco-polo} Let $(\A,\alpha)$  be a $C^*$-dynamical system,
$J$ an ideal orthogonal to the kernel of $\alpha$, and $(\B,\alpha)$
   a partially reversible system associated to the triple $(\A,\alpha,J)$ as described above.
   If the partial homeomorphism  $\widehat{\alpha}$ dual to $\alpha:\B\to \B$ is topologically free,
   then for  any  faithful $J$-covariant representations $(\pi,U,H)$ of $(\A,\alpha)$
    the epimorphism $\pi\times U: C^*(\A,\alpha, J) \to C^*(\pi(\A),U)$ given by \eqref{e-fath} is an isomorphism:
   $$
   C^*(\A,\alpha, J) \cong C^*(\pi(\A),U).
   $$
  \end{thm}
  A crucial and (in the noncommutative case) still open question  is the following.
\begin{quote}
     \textbf{Problem:} How to express the topological freeness of the partial reversible dynamical system $(\widehat{\B},\widehat{\alpha})$ dual to $(\B,\alpha)$ in terms of the initial $C^*$-dynamical system $(\A,\alpha)$ and the ideal $J$?
\end{quote}

\subsection{Crossed product overview}

One can see that the most popular   crossed products by endomorphisms coincide with   $C^*(\A,\alpha,J)$ for certain  $J$.   Table 1  presents the corresponding  juxtaposition of the  objects chosen.

\begin{table}[htb]
\begin{center}
        \begin{tabular}{|c|c|c|c|} \hline
            N. & endomorphism $\alpha:\A\to \A$ & $J \triangleleft \A$ & $C^*(\A,\alpha,J)$
            \\ \hline
            1. & automorphism    & $J=(\ker\alpha)^\bot=\A$                 & classical unitary
            \\ &                 &                                        & crossed product
            \\ \hline
            2. & monomorphism   & $J=(\ker\alpha)^\bot=\A$                 & isometric crossed product
            \\ &                &                                         &    \cite{Paschke}, \cite{Murphy}
            \\ \hline
            3. & $\ker\alpha$ unital and     & $J=(\ker\alpha)^\bot$                 & crossed product using
            \\ &  $\quad \alpha(\A)$ hereditary in $\A\quad $    &                     & complete transfer operator
            \\ &                                               &                     & \cite{Ant-Bakht-Leb}
            \\ \hline
          4. & $\ker\alpha$ unital and   & $J=(\ker\alpha)^\bot$                 & covariance algebra \cite{kwa}
            \\ & $\A$ commutative        &                                       &
            \\ \hline
            5. & arbitrary  & $J=\{0\}$                 & partial-isometric
            \\ &            &                           & crossed product   \cite{Lin-Rae}
            \\ \hline
            6. & arbitrary  & $ \{0\}\subset  J \subset (\ker\alpha)^\bot$       & partial-isometric
            \\ &            &                                                    & crossed product
             \\&            &                                                 & $C^*(\A,\alpha,J)$
             \\ &            &                                                 &of the present article,
            \\ \hline
        \end{tabular}
        \caption{Different crossed products}\label{table 1}
\end{center}
\end{table}

To see the coincidence in N.3 of Table 1 we refer the reader to
\cite[Prop. 2.6]{kwa4}. The crossed product N.6
(Definitions~\ref{crossed product defn} and~\ref{cr-pr-def}) is
the most general in the sense that it gives all the remaining ones
for an appropriate choice of $J$; namely $J=(\ker\alpha)^\bot$ for N.1-4
and $J=\{0\}$ for N.5.

As it is shown in \cite[Section 4]{Ant-Bakht-Leb} the crossed product N.3  of Table 1 covers a lot of most popular crossed product constructions, and in particular two kinds of  crossed products introduced by R. Exel in \cite{exel1} and   \cite{exel2}   may be obtained from  the crossed product N.3.

Note also that  there are a number of crossed products that at first sight are not of type $C^*(\A,\alpha,J)$
and are related to ideals that  are not orthogonal to $\ker\alpha$. As an example let us mention Stacey's (multiplicity one) crossed product \cite[Defn. 3.1]{Stacey}. For the sake of simplicity we state his definition in a unital setting, cf. \cite{Adji_Laca_Nilsen_Raeburn}.
\begin{defn}\label{Stacey}
\emph{Stacey's crossed product} for an
endomorphism $\alpha$ of a unital $C^*$-algebra $\A$ is a unital
$C^*$-algebra $B$ together with a unital $^*$-homomorphism
$i_A:A\to B$ and an isometry $u\in B$ such that
\begin{itemize}
\item[i)] $i_\A(\alpha(a))=u\, i_A(a)u^*$ for all $a\in \A$
\item[ii)] $B$ is generated by $i_A(A)$ and $u$
\item[iii)] for every non-degenerate representation $\pi:A\to L(H)$ and an isometry $T\in L(H)$ there is a representation $\pi\times T:B\to L(H)$ such that $(\pi\times T)\circ i_\A=\pi$ and $(\pi\times T)(u)=T$.
\end{itemize}
\end{defn}

The mapping $a\to uau^*$ for an isometry $u$ is injective.
Therefore, in view of  condition i), the homomorphism $i_\A$  can
not be injective unless  $\alpha$ is a monomorphism. Hence in
general $A$ does not embeds into the Stacey's crossed product.
In particular, see \cite[Prop. 2.2]{Stacey}, Stacey proved that $B$
degenerates to $\{0\}$ if and only if the inductive limit of
the inductive sequence
$\A\stackrel{\alpha}{\rightarrow}\A\stackrel{\alpha}{\rightarrow}
...$ degenerates to zero. We will refine this result in Example \ref{reduction
of Stacey's crossed product} below.
Actually we show  that Stacey's crossed product can also be  presented in
the form of $C^*(\A,\alpha,J)$ for an appropriate $\A,\alpha$ and
$J$, but to achieve this one first  needs to 'reduce' the initial
$C^*$-dynamical system. The general procedure of such a reduction is
discussed in the forthcoming part of the paper where we also
analyse  the relation between the crossed products
$C^*(\A,\alpha,J)$ and  relative Cuntz-Pimsner algebras.

\section{Crossed products and relative Cuntz-Pimsner
algebras}\label{C-P}

In this section we start to discuss  relations between the
crossed products introduced and  relative Cuntz-Pimsner algebras.
This requires a description of a series of known objects
and results and we do this job in Subsections \ref{preliminaries
on C-correspondences}, \ref{The Toeplitz C*-algebra of a Hilbert
bimodule} and \ref{Relative Cuntz-Pimsner algebras}, while a
presentation of the crossed products as  relative Cuntz-Pimsner
algebras is given in Subsection~\ref{Crossed prod =Relative
Cuntz-Pimsner algebras}. To begin with we recall  the basic
necessary objects related to Hilbert $C^*$-modules and $C^*$-correspondences. A general
information on Hilbert $C^*$-modules  can be found, for example, in \cite{lance}. The term $C^*$-correspondence was popularized, among the others, by T. Katsura \cite{katsura1}, \cite{katsura}, \cite{katsura2}.

\subsection{$C^*$-correspondences and their representations}\label{preliminaries on C-correspondences}
Let  $\A$ be a (not necessarily unital)   $C^*$-algebra and $X$ a
right Hilbert $\A$-module with an  $\A$-valued inner product
$\langle \cdot,\cdot \rangle_\A $, see \cite[ch. 1]{lance}. We denote  by
$\LL(X)$  the $C^*$-algebra of adjointable operators  on $X$. For
$x,y \in X$,  we let  ${\Theta}_{x,y}\in \LL(X)$ be the
'one-dimensional operator': ${\Theta}_{x,y}(z)=x \cdot \langle
y,z\rangle_\A$, and  we denote by $\KK(X)$  the ideal of 'compact
operators'
   in $\LL(X)$
which is a closed linear span of  the operators ${\Theta}_{x,y}$,
$x,y\in X$. We recall that any $C^*$-algebra $\A$ can be naturally
treated as a right Hilbert $\A$-module where $\langle a,
b\rangle_\A=a^*b$ and then $\KK(\A)=\A$ and $\LL(\A)=M(\A)$ is the
multiplier algebra of $\A$.

\begin{defn}
\emph{A $C^*$-correspondence $X$ over a $C^*$-algebra $\A$} is  a (right) Hilbert $\A$-module equipped with a homomorphism $\phi:\A \to \LL(X)$.
 We refer to $\phi$ as the left action of  a $\A$ on  $X$ and  write
\begin{equation}\label{left}
 a\cdot x := \phi(a)x.
 \end{equation}
\end{defn}
\begin{rem}
\label{Hilbbert-bim} A $C^*$-correspondence is also sometimes
called a Hilbert bimodule, see e.g. \cite{p}, \cite{fr},
\cite{fmr}. However, there are  plenty of reasons, see e.g.
\cite{aee}, \cite{katsura1}, \cite{kwa-doplicher} or \cite{kwa3},   that the term
\emph{Hilbert bimodule} should be reserved for a special sort of
$C^*$-correspondence, namely a $C^*$-correspondence $X$ with an
additional structure which is an $\A$-valued  sesqui-linear form
${_A\langle} \cdot , \cdot \rangle$ such that
$$
 x \cdot \langle y ,z \rangle_A = {_A\langle} x , y  \rangle \cdot z, \qquad \textrm{for all}\,\,\, x,y,z\in X.
$$
 Then, see  \cite[Lem 3.4]{katsura1} or \cite[Prop. 1.11]{kwa-doplicher},   $X$ is both a left and a right Hilbert $\A$-module and $\|\langle x ,x \rangle_A\| = \|{_A\langle} x , x  \rangle \|$.
  \end{rem}
  \begin{defn}
  A {\em representation} $(\pi,t,B)$ of a $C^*$-correspondence
$X$  in a
$C^*$-algebra
$B$ consists of a linear map $t:X\to B$
and a homomorphism $\pi:A\to B$ such that
\begin{equation}\label{c*-corr}
t(x\cdot a) = t(x)\pi(a),\quad t(x)^*t(y)= \pi(\langle
x,y\rangle_A),\quad t(a\cdot x) = \pi(a)t(x),
\end{equation}
for $x,y\in X$ and $a\in A$. If $\pi$ is faithful (then
automatically $t$ is isometric, cf.  \cite[Rem 1.1]{fr}) we  say
that the  representation  $(\pi,t,B)$  is \emph{faithful}. If
$B=L(H)$ for  a Hilbert space $H$ we say that   $(\pi,t,L(H))$ is
a \emph{representation of $X$ in} $H$.
\end{defn}
\begin{rem}
The above introduced notion  is called in \cite{fr}, \cite{fmr} a
\emph{Toeplitz representation} of $X$, and   in
\cite{ms} it is called an \emph{isometric covariant representation} of $X$.
\end{rem}
  Any representation $(\pi,t,B)$ of a $C^*$-correspondence $X$ in a $C^*$-algebra $B$
  naturally give rise to a representation  ${(\pi,t,B)}^{(1)}: \KK(X)\to B$  of the $C^*$-algebra $\KK(X)$ of 'compact operators'
  on $X$ which is uniquely determined by the condition that
\begin{equation}\label{induced representation on K(X)}
{(\pi,t,B)}^{(1)}({\Theta}_{x,y}):=t(x){t(y)}^*\,\,\,\text{ for } \,\,\,x,y\in
X,
\end{equation}
see \cite[Prop.~1.6]{fr} or \cite[Prop. 3.13]{kwa-doplicher}.
Moreover the left action $\phi:\A\to \LL(X)$ restricted to the
ideal
$$
J(X):={\phi}^{-1}(\KK(X))
$$
is a representation of $J(X)$ in $\KK(X)$ and it is of a
particular interest to understand the relationship between
$\pi:J(X)\to B$ and  $(\pi,t,B)^{(1)} \circ \phi:J(X)\to B$.
\begin{defn}
  For any    ideal $J$ contained in $J(X)$ a representation $(\pi,t,B)$ of $X$ is said to be {\em covariant on} $J$  or \emph{$J$-covariant} if
\begin{equation*}
{(\pi,t,B)}^{(1)}(\phi(a))=\pi(a)\quad\text{for all \ } a\in J.
\end{equation*}
Actually,   the set $ \{a\in J(X):
{(\pi,t,B)}^{(1)}(\phi(a))=\pi(a)\}$
 is  the biggest ideal on which $(\pi,t,B)$  is covariant and  we will call this ideal
  an \emph{ideal of covariance for} $(\pi,t,B)$.
\end{defn}
\begin{rem} What we call above covariant represenations was originally called by Muhly and Solel coisometric representations, cf.  \cite{ms}, \cite{fmr}. This name was motivated by specicif applications and examples, and therefore  we choose, following Katsura \cite{katsura1}, \cite{katsura},  more universal (neutral) name - covariance.
\end{rem}
\begin{rem}\label{orthogonality remark}
Plainly, for  \emph{$J$-covariant} representation $(\pi,t,B)$
the ideal $\ker\phi \cap J$ is contained in $\ker\pi$. Hence a
necessary condition for the existence of a faithful
\emph{$J$-covariant} representation of $X$ is that $J$ is
orthogonal to $\ker\phi$ (by Proposition \ref{injectivity of k_A}
below this condition is also sufficient).
 \end{rem}
\subsection{The Fock representation }\label{The Toeplitz C*-algebra of a Hilbert bimodule}

The original idea of Pimsner \cite{p}, see also \cite{ms}, \cite{fr}, \cite{katsura}, is to construct representations of $C^*$-correspondences by a natural adaptation of the  celebrated Hilbert space construction introduced by Fock. Namely, given a $C^*$-correspondence  $X$ over $\A$, for $n\ge 1$, the
$n$-fold internal tensor product $X^{\otimes n} :=
X\otimes_\A\dotsm\otimes_\A X$, see e.g. \cite[ch. 4]{lance}, is  a
$C^*$-correspondence where $\A$ acts on the left by
$$
\phi^{(n)}(a)(x_1\otimes_\A\dotsm\otimes_\A x_n) :=
(a\cdot x_1)\otimes_\A\dotsm\otimes_\A x_n;
$$
here $a\cdot x_1$ is given by \eqref{left}.
 For $n=0$, we take $X^{\otimes 0}$
to be the Hilbert module
$\A$ with left action
$\phi^{(0)}(a) b: = ab$. Then the Hilbert-module direct sum, see \cite[p. 6]{lance},
$$
\FF(X) :=
\bigoplus_{n=0}^\infty X^{\otimes n}
$$ carries a diagonal left action $\phi_\infty$
of $\A$ in which $\phi_\infty(a)(x):=\phi^{(n)}(a)x$ where  $x\in
X^{\otimes n}$.  The
$C^*$-correspondence $\FF(X)$ is called
 the \emph{Fock space} over  the $C^*$-correspondence  $X$.
For each  $x\in X$, we
define a \emph{creation operator\/}
$T(x)$  on $\FF(X)$ by
$$
T(x)y=\begin{cases}
 x\cdot y
    & \text{if $y\in X^{\otimes 0}=\A$} \\
  x\otimes_\A y
    & \text{if $y\in X^{\otimes n}$ for some $n\geq 1$;} \\
\end{cases}
$$
routine calculations show that $T(x)$ is adjointable and its adjoint is the  \emph{annihilation operator}
$$
T(x)^*z=\begin{cases}
 0
    & \text{if $z\in X^{\otimes 0}=\A$} \\
  \langle x, x_1\rangle_\A\cdot y
    & \text{if $z=x_1\otimes_\A y\in X\otimes_\A X^{\otimes n-1}=
X^{\otimes n}$.}
\\
\end{cases}
$$
One sees  that $T:X\to\LL(\FF(X)) $ is an injective linear mapping and since $\A$ is a summand of $\FF(X)$, the map
 $\phi_\infty:\A \to \LL(\FF(X))$ is injective as well. Actually,  $(\phi_\infty,T, \LL(\FF(X)))$ is a
 faithul representation of $X$ whose ideal of covariance is $\{0\}$. 
\begin{defn}
The  representation   $(\phi_\infty, T, \LL(\FF(X)))$  of $X$ in
$\LL(\FF(X))$ defined above  is called \emph{Fock representation}
and the \emph{Toeplitz $C^*$-algebra} $\TT(X)$ of $X$ is by
definition the $C^*$-subalgebra of $\LL(\FF(X))$ generated by
$\phi_\infty(\A)\cup T(X)$, cf. \cite[Def. 2.4]{ms}, \cite[Def.
1.1]{p}.
 \end{defn}

\subsection{Relative Cuntz-Pimsner algebras $\OO(X,J)$}\label{Relative Cuntz-Pimsner algebras}
Composing the Fock representation $(\phi_\infty, T, \LL(\FF(X))$
of $X$ with the  quotient map one can get a representation of  $X$
whose ideal of covariance is an arbitrarily chosen ideal
contained in $J(X)={\phi}^{-1}(\KK(X))$. Namely, let  $J$ be an
ideal in $J(X)$ and let $P_0$ be the projection in $\LL(\FF(X))$
 that maps $\FF(X)$ onto the first summand $X^{\otimes 0}=\A$.
 One can  show \cite[Lem. 2.17]{ms} that $\phi_\infty(J)P_0$  is contained in $\TT(X)$.
 We will  write $\JJ(J)$ for the ideal in $\TT(X)$ generated by $\phi_\infty(J)P_0$.

 \begin{defn}[\cite{ms}, Def. 2.18]
If $X$ is a $C^*$-correspondence  over a $C^*$-algebra $\A$, and
if $J$ is an ideal in $J(X)={\phi}^{-1}(\KK(X))$ we denote by
$\OO(J,X)$ the quotient algebra $\TT(X)/\JJ(J)$ and call it
\emph{relative Cuntz-Pimsner algebra} determined by $J$.
 \end{defn}
The  algebras $\OO(J,X)$ are characterized by the following universal property.
\begin{prop}[\cite{fmr}, Prop. 1.3]\label{RCP algebra}
Let $X$ be a $C^*$-correspondence over $\A$, and let $J$ be an
ideal in $J(X)$. Let $q:\TT(X)\to \OO(J,X)$ be the quotient map
and put
$$
 k_\A= q\circ \phi_\infty  \quad\textrm{ and } \quad k_X=q \circ T.
$$
 Then  $(k_\A,k_X, \OO(J,X))$ is  a representation of $X$ which is covariant on $J$ and satisfies:
\begin{itemize}
\item[(i)] for every  representation $(\pi,t)$ of $X$ which is covariant on $J$, there is a homomorphism $\pi{\times}_J t$ of $\OO(J,X)$ such that
$$(\pi{\times}_J t)\circ k_\A=\pi \quad \textrm{  and } \quad (\pi{\times}_J t)\circ k_\A=t,$$
\item[(ii)] $\OO(J,X)$ is generated as a $C^*$-algebra by $k_X(X)\cup k_\A(\A)$.
\end{itemize}
The representation $(k_\A,k_X,\OO(J,X))$ is unique in the following sense: if $(k_\A',k_X',B)$ has similar properties, there is an isomorphism $\theta:\OO(J,X)\to B$ such that $\theta\circ k_X=k_X'$ and $\theta\circ k_\A=k_\A'$. In particular, there is a strongly continuous gauge action $\gamma:\T\to\Aut\OO(J,X)$ where  ${\gamma}_z(k_\A(a))=k_\A(a)$ and ${\gamma}_z(k_X(x))=zk_X(x)$ for $a\in \A,x\in X$.
\end{prop}
\begin{rem}\label{ideal of covariance vs universality}
One may show, cf. \cite[Prop. 4.7]{kwa-doplicher}, that the ideal of
coisomtetricity for the universal representation $(k_X,k_\A,
\OO(J,X))$ coincides with $J$. Hence $(k_X,k_\A, \OO(J,X))$ could
be equivalently defined as a universal triple with respect to
representations whose ideal of covariance not only contains
but actually equals $J$.
\end{rem}
\begin{table}[hbt]
\begin{center}
        \begin{tabular}{|c|c|c|c|} \hline
            N. &  $\phi:A\to \L(X)$ & $J \triangleleft J(X)$ & $\OO(J,X)$
            \\ \hline
            1. & monomorphism    & $J=J(X)$                 & Cuntz-Pimsner algebra of $X$
            \\ &                 &                                        & \cite{p}, \cite{fmr}
            \\ \hline
            2. & arbitrary   & $ J=\{0\}$                 & Toeplitz algebra of $X$
            \\ &                &                                         &    \cite{fr}, \cite{fmr}
            \\ \hline
            3. & arbitrary     & $J=(\ker\phi)^\bot\cap J(X)$                 & Katsura's algebra of $X$
          \\
               &    &   & \cite{katsura1}, \cite{katsura}, \cite{katsura2}
                       \\ \hline
          4. & $\phi(J)=\KK(X)$   & $J=(\ker\phi)^\bot\cap J(X)$    & crossed product by
          \\
              &                                  &                      &   the Hilbert bimodule
           \\
               &    &   & \cite{aee}
                       \\ \hline
        \end{tabular}
        \caption{Different  relative Cuntz-Pimsner algebras \label{table 2}}
\end{center}
\end{table}
Table \ref{table 2}  presents a   juxtaposition of various relative Cuntz-Pimnser algebras studied in the literature obtained from $\OO(J,X)$ for different choice of the ideal $J$. To see the coincidence in N.4 of Table \ref{table 2} we refer the reader, for instance, to \cite{katsura1}. In view of the following proposition,  the  algebra $\A$ embeds into $\OO(J,X)$ for all the algebras presented in the Table \ref{table 2}.

\begin{prop}[\cite{ms}, Prop.~2.21]\label{injectivity of k_A}
Let $X$ be a $C^*$-correspondence  over $\A$ and let
$(k_X,k_\A,\OO(J,X))$ be the  relative Cuntz-Pimsner algebra
associated with $J$. Then $k_\A:\A\to\OO(J,X)$ is injective if and
only if
\begin{equation}\label{orthogonal property}
\ker\phi \cap J =\{0\}.
\end{equation}
In particular, cf. Remark \ref{orthogonality remark}, a faithful $J$-covariant representation of $X$ exists if and only if $J$ is orthogonal to $\ker\phi$.
\end{prop}
The analysis presented in the forthcoming sections will show, in
particular, that  all the relative Cuntz-Pimsner algebras are in
fact the algebras $\OO(J,X)$ where $J\subset (\ker\phi)^{\bot}$
(see Theorem~\ref{reduction thm}). Moreover, see Theorem
\ref{Katsuras canonical theorem} below, they all can be modeled by
the  algebra N.3 in  Table 2 or even by the algebra N.4, see
\cite[Thm. 3.1]{aee}, but in the latter case the passage is highly
non-trivial (see also Remark~\ref{6.7}).

\subsection{Crossed products as relative Cuntz-Pimsner algebras}\label{Crossed prod =Relative Cuntz-Pimsner algebras}

Let $(\A,\alpha)$ be a $C^*$-dynamical system  and define the structure of a  $C^*$-correspondence  over $\A$ on the space
$$
X:=\alpha(1) \A
$$
 by
\begin{equation}
\label{e-bim-a1}
a \cdot x :=\alpha(a)x,
\qquad
x\cdot a:= xa,
\qquad
\langle x,y\rangle_\A:=x^*y.
\end{equation}
Then one easily  checks that $J(X)=\A$  and  $\ker\alpha =\ker\phi$. We will say that $X$ is the \emph{$C^*$-correspondence of the $C^*$-dynamical system}  $(\A,\alpha)$. The proof of the foregoing proposition  in essence follows the  argument from \cite[Exm. 1.6]{fmr}.
\begin{prop}\label{universality proposition}
Let $X$ be  the $C^*$-correspondence of a $C^*$-dynamical system
$(\A,\alpha)$. The relations
\begin{equation}\
\label{e-bim-a4} U:=t(\alpha(1))^*, \qquad { and}\qquad t(x):=
U^*\pi(x)
\end{equation}
establish a one-to-one correspondence between representations
$(\pi,U,H)$  of $(\A,\alpha)$ and  representations $(\pi,t,L(H))$
of $X$, under this correspondence the  ideal  of covariance
for $(\pi,t,L(H))$ coincides with  the ideal of covariance for
$(\pi,U,H)$.
\end{prop}
\begin{Proof} Let $(\pi,t,L(H))$ be a  representations of  $X$ and  put  $U:=t(\alpha(1))^*$.
Then exploiting \eqref{c*-corr} and \eqref{e-bim-a1} one gets
$$
\pi(\alpha(a))=\pi(\langle \alpha(1),\alpha(a)\rangle_\A)=t(\alpha(1))^*t(\alpha(a))=U t(a \cdot \alpha(1))=U\pi(a)U^*.
$$
 Thus, $(\pi,U, H)$ is a  representation of $(\A,\alpha)$.
  Moreover, observe that the operator $\phi(a)$ is just $\Theta_{\alpha(a),\alpha(1)}$ and since
$$
{(\pi,t,L(H))}^{(1)}(\Theta_{\alpha(a),\alpha(1)})=t(\alpha(a))t(\alpha(1))^*=U^*\pi(\alpha(a))U
=U^*U\pi(a)U^*U=U^*U\pi(a)
$$
it follows that the ideal of covariance for  $(\pi, t, L(H))$
coincides with the ideal of covariance for   $(\pi,U, H)$.

Conversely, for any    covariant representation  $(\pi,U,H)$ of
$(\A,\alpha)$  putting $t(x):=U^*\pi(x)$, $x\in X$, we have
$U=t(\alpha(1))^*$ and one easily checks    conditions
\eqref{c*-corr}.
\end{Proof}
By the universality of $\OO(J,X)$ and $C^*(\A,\alpha,J)$ (recall Definition~\ref{crossed product defn}), see also  Remark \ref{ideal of covariance vs universality}, we get the following
\begin{cor}\label{C-P-cross}
Let $X$ be  the  $C^*$-correspondence of a $C^*$-dynamical system
$(\A,\alpha)$ and let $J$ be an ideal in $(\ker\alpha)^\bot$. Then
algebras $\OO(J,X)$ and $C^*(\A,\alpha,J)$ are canonically
isomorphic. In particular,
\begin{itemize}
\item[(i)] $\OO(J,X)$ is generated as a $C^*$-algebra by the partial isometry $u=k_X(\alpha(1))^*$ and the $C^*$-algebra $k_\A(\A)$.
\item[(ii)] for every $J^\prime$-covariant  representation $(\pi,U, H)$ of $(\A,\alpha)$ with $J^\prime \supset J$, there is a homomorphism $\pi{\times}_J U$ of $\OO(J,X)$ uniquely determined by
$$
(\pi{\times}_J U)(u):= U\qquad \textrm{ and  }\qquad
(\pi{\times}_J U)\circ k_\A :=\pi.
$$
\end{itemize}
\end{cor}
\begin{cor}\label{corollary stacey's}
Let $X$ be  the  $C^*$-correspondence of a $C^*$-dynamical system
$(\A,\alpha)$.  Then the algebra $\OO(\A,X)$ together with the mapping
$k_\A$ and operator $u:=k_X(\al(1))^*$, cf. Proposition \ref{RCP algebra}, forms a Stacey's
crossed product for $(\A,\alpha)$ (Definition~\ref{Stacey}).
\end{cor}
Note that $\OO(J,X)$ is defined for ideals $J$ that are not
necessarily orthogonal to~$\ker\alpha$. This along with
Proposition~\ref{universality proposition} may make one   guess  that
$\OO(J,X)$ is a more general object than  $C^*(\A,\alpha,J)$. At
the same time Proposition~\ref{injectivity of k_A} shows that when
$J$ is not  orthogonal to $\ker\alpha$ the algebra  $\OO(J,X)$
possesses certain 'degeneracy'. All this stimulates us to take a
closer look and provide a more thorough analysis of the structure
of $\OO(J,X)$ and its relation to $C^*(\A,\alpha,J)$. This is the
theme of the next section, where  we  present the  procedure of
the canonical reduction of $C^*$-correspondences,  algebras and
$C^*$-dynamical systems.  In particular, as a result of this
reduction we  establish the coincidence of $\OO(J,X)$ with
appropriate crossed products introduced in the article.

\section{Reductions of $C^*$-correspondences
}\label{Reduction and canonical}

\subsection{Reduction of $C^*$-correspondences}\label{Reduction of Hilbert bimodule}

Let us now fix a $C^*$-correspondence $X$ over $\A$ and an ideal
$J$ in $J(X)$. We will  reduce $X$ by taking quotient  to
a certain 'smaller' $C^*$-correspondence satisfying
\eqref{orthogonal property} and yielding  the same relative
Cuntz-Pimsner algebra as $X$ and $J$.

To this end we note that $C^*$-correspondences behave nice under
quotients. Namely, if $I$ is an ideal in $A$, then
$XI:=\clsp\{xi:x\in X,\, i\in I\}$ is  both a right Hilbert
$A$-submodule of $X$ and  a  right Hilbert $I$-module, as  we
have
\begin{equation}\label{XI equation}
 XI=\{xi:x\in X,\, i\in I\}=\{x\in X: \langle x,y\rangle_A\in I \textrm{ for all }y\in X\},
 \end{equation}
cf. \cite[Prop. 1.3]{katsura2}. Moreover, we may consider the quotient  space $X/XI$ as a right Hilbert $\A/I$-module with an $\A/I$-valued inner product and right action of $\A/I$ given by
\begin{equation}\label{right action}
\langle q_{XI}(x),q_{XI}(y)\rangle_{\A/I}:=q_{I}(\langle x,y\rangle_\A), \qquad q_{XI}(x)\cdot q_I(a):=q_{XI}(x\cdot a),
\end{equation}
where $q_I:\A\to \A/I$ and $q_{XI}: X \to X/ XI$ are the quotient
maps, cf. \cite[Lem. 2.1]{fmr}. However, in order to define a left
action on the quotient $X/XI$ we need to impose the following
condition on an ideal $I$ in $\A$:
\begin{equation}\label{X-invariance}
\phi(I)X\subset XI.
\end{equation}
An ideal $I$ in $\A$ satisfying \eqref{X-invariance} is called  $X$-\emph{invariant}
 and for such an ideal the quotient $A/I$-module $X/XI$
is equipped with quotient left  action $\phi_{\A/I}:\A/I\to
\LL(X/XI)$ given by
\begin{equation}\label{left action}
\phi_{\A/I}(q_{I}(a))q_{XI}(x):=q_{XI}(\phi(a)x),\qquad x \in
X,\,\, a\in \A.
\end{equation}
and hence $X/XI$ naturally becomes a $C^*$-correspondence  over
$\A/I$, cf.  \cite[Lem. 2.3]{fmr}, \cite{katsura2}. We refer to it
as to a \emph{quotient $C^*$-correspondence}.

We will apply the following main result of \cite{fmr} to the reduction ideal introduced below.

\begin{thm}[\cite{fmr}, Thm. 3.1] \label{takie tam aa}
Suppose $X$ is a $C^*$-correspondence over $\A$, $J$ is an ideal
in $J(X)$, and $I$ is an $X$-invariant ideal in $\A$. If we denote
by  $\I(I)$ the  ideal in $\OO(J,X)$ generated by $k_\A(I)$, then
the quotient $\OO(J,X)/\I(I)$ is canonically isomorphic to
$\OO(q_I(J),X/XI)$.
\end{thm}

\begin{defn}\label{reduction ideal for correspondences}
For any ideal $J$ in $J(X)$ we define recursively an  increasing sequence of ideals
\begin{equation}\label{J-n}
J_0:=\{0\}\,\,\, \textrm{ and }\,\,\,   J_{n+1}:=\{a\in J:
\phi(a)X\subset X{J_n}\}\,\,\, \textrm{ for } n\geq 0,
\end{equation}
and call the ideal
$$
J_\infty:=\overline{\bigcup_{n\in \N} J_n}
$$
the \emph{reduction ideal} for $X$ and $J$,
\end{defn}
\begin{rem}\label{universal description of J_infty}
Since $\phi(J_{n+1})X\subset X J_n$ and $J_n\subset J_{n+1}$,  the ideals $J_n$, $n\in \N$,  and therefore also $J_\infty$ are $X$-invariant  ideals in $\A$. Actually,
let us note that
\begin{equation}\label{I_infty condition}
a \in J\,\, \textrm{ and } \,\, \phi(a)X\subset X J_\infty \Longrightarrow  a\in J_\infty,
\end{equation}
and  this implication characterizes  $J_\infty$ in the sense that $J_\infty$ is the smallest $X$-invariant ideal in $\A$ satisfying \eqref{I_infty condition}. In particular, $J_\infty=\{0\}$ if and only if $\ker \phi\cap J=\{0\}$.
Note also that  $C^*$-correspondences  $X / XJ_{n}$, $n\in \N$, may
be considered as  'approximations' of $X / XJ_{\infty}$, and  if $J_n=J_{n+1}$, for certain $n\in \N$, then
$J_\infty=J_n$ and $X / XJ_{n}=X / XJ_{\infty}$.
\end{rem}

The next  result  states, in particular,  that the quotient
$C^*$-correspondence $X/XJ_{\infty}$ and the quotient
$C^*$-algebra  $\A/J_\infty$ may be identified with the image of
the initial $C^*$-correspondence $X$ and $C^*$-algebra $\A$ in the
relative Cuntz-Pimsner algebra $\OO(J,X)$.
\begin{thm}\label{reduction thm}
Let  $X$ be a $C^*$-correspondence over $\A$ and  $J$  an ideal in
$J(X)$. Then for $n\in \N\cup\{\infty\}$, we have a canonical
isomorphism
\begin{equation}\label{O-n}
\OO(J,X) \cong \OO(q_{J_n}(J),X/XJ_{n}).
\end{equation}
Moreover, for $n=\infty$ we have
\begin{equation}\label{ker-infty}
\ker \phi_{\A/J_\infty}\cap q_{J_\infty}(J)=\{0\}
\end{equation}
and thus we have the following (again canonical) isomorphisms
\begin{equation}\label{iso-infty}
k_\A(\A)\cong \A/J_\infty, \qquad k_X(X)\cong X/XJ_\infty.
\end{equation}
\end{thm}
\begin{Proof} In view of Theorem  \ref{takie tam aa} to prove the first part of theorem it is enough to show that  for every ideal $J_n$, $n=0,1,..., \infty$,  we have  $k_\A(J_n)=0$.  It is  clear that $k_\A(J_0)=0$. Assume that $k_\A(J_n)=0$ and let $a\in J_{n+1}$. Then for every $x\in X$ there exists $y(x)\in X$ and $i(x)\in J_n$ such that $\phi(a)x=y(x)i(x)$ and thus
$$
k_\A(a)k_X(x)=k_X(\phi(a)x)= k_X(y(x)i(x))=k_X(y(x))k_\A(i(x))=0,
$$
that is $k_\A(a)k_X(X)=\{0\}$.
Moreover, since $J_{n+1}\subset J$, and the universal representation is $J$-covariant we have $k_\A(a)={(k_\A,k_X,\OO(J,X))}^{(1)}(\phi(a))$, for the mapping given by \eqref{induced representation on K(X)}. By  \eqref{induced representation on K(X)}, relation ${(k_\A,k_X,\OO(J,X))}^{(1)}(\phi(a))k_X(X)=\{0\}$ imply ${(k_\A,k_X,\OO(J,X))}^{(1)}(\phi(a))=0$ and therefore
$k_A(a)=0$. Hence $k_\A(J_{n+1})=0$, and it follows that $k_\A(J_n)=0$ for every
$n=0,1,..., \infty$.
\\
 To prove that  $ \ker \phi_{\A/J_\infty}\cap
q_{J_\infty}(J)=\{0\} $ take $a\in J$ and suppose that
$\phi_{\A/J_\infty}( q_{J_\infty}(a))=0$.  Then by  \eqref{left
action}  we see that $\phi(a) X \subset XJ_{\infty}$ and  by
\eqref{I_infty condition} we have $q_{J_\infty}(a)=0$.
 Now it suffices to apply Proposition \ref{injectivity of k_A}.
\end{Proof}
\begin{rem}
 Theorem \ref{reduction thm} shows, in particular, that the ideal $J_\infty$ plays the role of a certain  'measure' of the degree of degeneracy of $\OO(J,X)$ --  the bigger $J_\infty$ is
  the smaller $\OO(J,X)$ is.
 In particular, $\OO(J,X)=0$ if and only if $J_\infty=
\A$. Obviously,  $\A$ embeds into $\OO(J,X)$ if and only if
$J_\infty=0$ which is  equivalent to $X=X/XJ_{\infty}$.
Moreover, it follows that  one may always  restrict his interest only to the relative  Cuntz-Pimsner algebras $\OO(J,X)$ determined by ideals such that
$$
 J \subset (\ker \phi)^\bot,
$$
since otherwise one has to pass (either explicitly or implicitly)
to the \emph{reduced $C^*$-correspondence} $X/XJ_\infty$ over the
\emph{reduced $C^*$-algebra} $A/J_{\infty}$ and the \emph{reduced
ideal} $q_{J_\infty}(J)\subset (\ker \phi_{\A/J_\infty})^\bot$.
\end{rem}
\subsection{Katsura's canonical relations for relative
Cuntz-Pimsner algebras}\label{Katsura's canonical relations}

 In \cite{katsura2} T. Katsura  associated with a $C^*$-correspondence $X$  a $C^*$-algebra $\OO_X$    which is the relative Cuntz-Pimsner algebra $\OO(J, X)$ for  $J=(\ker \phi)^\bot\cap J(X)$. By subtle  analysis  he proved that any relative  Cuntz-Pimsner algebra $\OO(J, X)$ is isomorphic to $\OO_{X_\omega}$ for certain 
 $X_\omega$.
In this subsection we apply the
reduction procedure presented above to give an alternative
definition  of these
$C^*$-cor\-res\-pon\-den\-ces.

The reduction ideal $J_\infty$ defined above
coincides with the ideal  $J_{-\infty}$ constructed in \cite[Sect.
11]{katsura2}. It is noted in \cite{katsura2}, that for the pair
$\omega=(J_\infty,J)$ there is a natural  (yet a bit
sophisticated)
 $C^*$-correspondence structure on the following pair  of 'pullbacks'
$$
A_\omega=\{(a,a')\in \A/J_\infty \oplus \A/J:
q_{J(J_\infty)}(a)=q_{J(J_\infty)}(a')\}
$$
$$
X_\omega=\{(x,x')\in X/X{J_\infty}\oplus X/XJ: q_{XJ(J_\infty)}(x)=q_{XJ(J_\infty)}(x')\}
$$where $J(J_\infty):=(q_{XJ_\infty})^{-1}((\ker \phi_{\A/J_\infty})^\bot \cap J(X/XJ_\infty))$  (then $J_\infty \subset J\subset J(J_\infty)$) and $q_{J(J_\infty)}$   denotes here the quotient map composed with a corresponding isomorphism $(\A/J_\infty)/J(J_\infty)\cong \A/J(J_\infty)$ or $(\A/J)/J(J_\infty)\cong \A/J(J_\infty)$ and similar convention is used for  $q_{XJ(J_\infty)}$.
Let us note that  the mappings
$$
\A/J_\infty \ni a \mapsto a\oplus q_J(a) \in \A_\omega, \qquad X/XJ_\infty \ni x \mapsto x\oplus q_J(x) \in X_\omega
$$
are injective  and therefore the  pair $(\A_\omega,X_\omega)$ is
an extension of the reduced pair $(\A/J_\infty, X/XJ_\infty)$ (it
is a nontrivial extension unless $J_\infty=J(J_\infty)$).

Thus
Katsura's construction in fact involves two steps -- reduction (in the sense of Theorem \ref{reduction thm}) and an
extension (which we describe precisely in Definition \ref{Katsura's relations} below); and perhaps exposing them  explicitly  makes his construction a
bit more transparent.
\begin{defn}\label{Katsura's relations}
Let $J$ be an ideal in $J(X)$ and  $J_X=(\ker \phi)^\bot \cap J(X)$. If $J$ is orthogonal to $\ker\phi$, then  \emph{Katsura's canonical $C^*$-correspondence for} $X$ and $J$ is  a $C^*$-correspondence   $X_\omega$  over $\A_\omega$
where
$$
A_\omega=\{(a,a')\in \A \oplus \A/J: q_{J_X}(a)=q_{J_X}(a')\}
$$
$$
X_\omega=\{(x,x')\in X\oplus X/XJ: q_{XJ_X}(x)=q_{XJ_X}(x')\}
$$
and  operations are defined  simply by coordinates:
$$
(x,x') \cdot (a,a')= (xa,x'a'),\qquad (a,a')\cdot  (x,x') = (ax,a'x')
$$
and
$$
\langle (x,x'), (y,y')\rangle_{\A_\omega} =(\langle x,y \rangle_\A, \langle  x',y'\rangle_{\A/J}) \in A_\omega.
 $$
 \end{defn}
 Extending the procedure presented above to the general situation
 we naturally arrive at the following
 \begin{defn}\label{Katsura's relations-general}
If $J$ is arbitrary (not necessarily orthogonal to $\ker\phi$), we define a pair $(\A_\omega,X_\omega)$ to be
Katsura's canonical $C^*$-correspondence for the reduced
$C^*$-correspondence $X/XJ_\infty$ and the reduced ideal
$q^{J_\infty}(J)$ and also call it \emph{Katsura's canonical
$C^*$-correspondence for} $X$ and $J$.
\end{defn}
\begin{rem}

 By Remark~\ref{universal description of J_infty} the coincidence of  notation and terminology in Definitions~\ref{Katsura's relations} and
 \ref{Katsura's relations-general} does not cause
confusion:  if $J$ is orthogonal to $\ker\phi$, then
$(\A_\omega,X_\omega)$ is an extension of $(A,X)$. Moreover
$(\A_\omega,X_\omega)\cong(A,X)$ if and only if $J=(\ker
\phi)^\bot \cap J(X)$.
\end{rem}
One can see that  the above Definition~\ref{Katsura's relations-general}
coincides with Katsura's construction and therefore  by
\cite[Prop. 11.3]{katsura2} we have
\begin{thm}\label{Katsuras canonical theorem}
If $J$ is an ideal in $J(X)$ and $(\A_\omega,X_\omega)$ is Katsura's canonical $C^*$-correspondence for $X$ and $J$, then we have an isomorphism
$$
\OO(J,X)\cong \OO((\ker \phi_{\omega})^\bot\cap J(X_\omega),
X_\omega)
$$
where $\phi_{\omega}$ denotes the left action of $\A_\omega$ on $X_\omega$.
\end{thm}

Additional discussion of Katsura's canonical $C^*$-correspondence
and its relation to the crossed product  will be given
in~\ref{canonic-c*-dynam-syst}.


\subsection{Reduction of $C^*$-dynamical systems}\label{Reduction of C*-dynamical systems}

Once we have noted the relation between crossed products and
relative Cuntz-Pimsner algebras in Subsection~\ref{Crossed prod
=Relative Cuntz-Pimsner algebras} and established  the reduction
procedure for relative Cuntz-Pimsner algebras in  Subection
\ref{Reduction of Hilbert bimodule}  it is reasonable to apply
this procedure to  crossed products, and this is the theme of
the present subsection. As a by product we also establish the
coincidence of  $\OO(J,X)$ with appropriate crossed products
introduced in the article.

 Let $X$ be the $C^*$-correspondence of a $C^*$-dynamical system $(\A,\alpha)$ and  let $J$ be an ideal in $\A$.
 We note that an ideal $I$ satisfies \eqref{X-invariance} if and only if  $\al(I)\subset I$ hence $X$-invariance
 is equivalent to $\al$-invariance. In particular, the  ideals $J_n$, $n=0,1,...,\infty$,
 from Definition \ref{reduction ideal for correspondences}  are $\al$-invariant.  Formulae  \eqref{J-n} mean that
\begin{equation}
\label{e-j-n}
J_n=\underbrace{\alpha^{-1}\Big(\alpha^{-1}\big(...(\alpha^{-1}}_{n\textrm{ times}}(\{0\})\cap J)...\big)\cap J\Big)\cap J , \,\,\,\qquad n\in \N,
\end{equation}
that is
$$
J_n=(\ker \alpha^{n})\cap \bigcap_{k=0}^{n-1} \alpha^{-k}(J).
$$
\begin{rem}
\label{remark-inf}
 Recalling Remark \ref{universal description of
J_infty} we note that $J_\infty=\overline{\bigcup_{n\in \N} J_n}$
is the smallest ideal in $A$ such that
$$
a \in J\,\, \textrm{ and } \,\, \al(a)\in J_\infty \Longrightarrow
a\in J_\infty,
$$
and $J_\infty=\{0\}$ if and only if $ \ker\alpha \cap J=\{0\}$.
\end{rem}

We give an alternative description of $J_\infty$ in the following lemma.
\begin{lem}\label{lemma about J_infty}
The sets  $\overline{\{a \in \A: \exists_{n\in\N}\,\, \alpha^n(a)=0\}}$ and $\{a \in \A: \lim_{n\to \infty} \alpha^n(a)=0\}$ coincide and form the smallest $\alpha$-invariant ideal $I_\infty$  in $\A$ such that $\alpha$ factors through to a monomorphism on the  quotient algebra $\A/I_\infty$. In particular,
\begin{equation}\label{ideal I_infinity}
J_\infty=\{a \in \A:  \alpha^n(a)\in J \textrm{ for all } n\in \N  \textrm{ and } \lim_{n\to \infty} \alpha^n(a)=0\}
\end{equation}
is the largest $\alpha$-invariant ideal contained in  $J\cap I_\infty$.
\end{lem}
\begin{Proof}
 One sees that both $I_1=\overline{\{a \in \A: \exists_{n\in\N}\,\, \alpha^n(a)=0\}}$ and $I_2:=\{a \in \A: \lim_{n\to \infty} \alpha^n(a)=0\}$ are $\alpha$-invariant ideals in $\A$ such that $\alpha$ factors through to monomorphisms both on $\A/I_1$ and $\A/I_2$. Moreover, it is clear that  $I_1$ is  the minimal ideal with the aforementioned properties. In particular, $I_1\subset I_2$ and to see the opposite inclusion note that since $a\mapsto \alpha(a)$ factors through to the monomorphism (isometric mapping) on $I_2/I_1$ and $\alpha^n(a)\to 0$ for all $a\in I_2$ it follows that $I_2/I_1=\{0\}$.

 For the second part of the assertion  note that
 $ a\in
\bigcup_{n\in \N} J_n$ if and only if there is $n=1,2,...,$ such that
$$
\al^n(a)=0, \qquad \al^{k}(a)\in J, \qquad k=1,2,...,n-1.
$$
Hence  $\bigcup_{n\in \N} J_n=(\bigcup_{n=0}^\infty\ker \alpha^{n})\cap \bigcap_{n=0}^{\infty}\al^{-n}(J)
 $ and thus $J_\infty=I_\infty\cap \bigcap_{n=0}^{\infty}\al^{-n}(J)$.
\end{Proof}
\smallskip

Let $n\in \N\cup\{\infty\}$. Since the ideal $J_n$ is
$\al$-invariant,  we have a quotient $C^*$-dynamical system
$(\A/J_n, \al_n)$  where  $\alpha_n:\A/J_{n} \to \A/J_n$ is  given
by $ \alpha_n\circ q_{J_n} =q_{J_n}\circ \alpha $ and the quotient
$C^*$-correspondence $X/X J_n$ may be viewed as the
 $C^*$-correspondence of $(\A/J_n,
\alpha_n)$. In particular
\begin{equation}\label{delta_infinity}
\alpha_\infty(a+ J_{\infty}):=\alpha(a) + J_\infty
\end{equation}
is an endomorphism of $\A/J_{\infty}$ such that
\begin{equation}\label{J-infty}
\ker \alpha_\infty \cap q_{J_{\infty}}(J)=\{0\}.
\end{equation}

  Obviously one may apply  Theorem \ref{reduction thm} to  each of the systems $(\A/J_n,\alpha_n)$, $n\in \N\cup\{\infty\}$, however,
  we  focus on the case $n=\infty$. Then by virtue of  Theorem~\ref{reduction thm} and Corollary~\ref{C-P-cross} along
  with \eqref{J-infty} we get
\begin{prop}\label{reducing C*-Hilbert bimodules}
If  $X$ is the  $C^*$-correspondence of a $C^*$-dynamical system
$(\A,\alpha)$ and  $J$ is an ideal in $\A$, then
$$\OO(J,X)=\OO( q_{J_\infty}(J), X/XJ_\infty)= C^* (\A/J_{\infty}, \alpha_\infty,q_{J_\infty}(J))$$
is  a universal algebra generated by a copy of the algebra $\A/J_{\infty}$ and a partial isometry $u$ subject to relations
$$
ua u^*=\alpha_\infty(a),\,\,\, a\in \A/J_{\infty}, \qquad q_{J_\infty}(J)=\{a\in \A/J_{\infty}: u^*u a=a\}.
$$
\end{prop}

\begin{cor}\label{kernel of a covariant representation}
If $(\pi,U, H)$ is a $J$-covariant representation of
$(\A,\alpha)$, then $ J_\infty \subset \ker \pi. $
\end{cor}

\begin{ex}\label{reduction of Stacey's crossed product}
Let us apply the above results  to  Stacey's crossed product
(Definition~\ref{Stacey}), and in particular,  obtain a refinement
of \cite[Prop. 2.2]{Stacey}.
 Note that if   $J=A$,  then
 $$J_\infty=\{a\in \A: \lim_{n\to \infty}\alpha^n(a)=0\}=\overline{\bigcup_{n=0}^\infty \ker\al^n},$$ cf. Lemma \ref{lemma about J_infty}.
Accordingly, the system $({\A_\infty},{\alpha_\infty})$, $\A_\infty:=\A/J_\infty$,  obtained by  the quotient of $(\A,\alpha)$ by  $J_\infty$
 could  be considered as the largest subsystem of $(\A,\alpha)$ with the property that ${\alpha_\infty}:{\A_\infty}\to {\A_\infty}$
 is a monomorphism. Moreover, by
 Corollary \ref{corollary stacey's} and Proposition \ref{reducing C*-Hilbert bimodules},  Stacey's crossed product is a universal $C^*$-algebra generated by a copy of $A_\infty$ and an isometry $u$ such that  $ua u^*=\alpha_\infty(a)$, for all $a\in A_\infty$. In particular, the Stacey's crossed product is the crossed product
 $C^*(A_\infty,\alpha_\infty,A_\infty)$ studied in the present paper and it  reduces to the zero algebra if and only if $\A=\{a\in \A: \lim_{n\to \infty}\alpha^n(a)=0\}$.
\end{ex}


\section{Canonical $C^*$-dynamical systems}\label{Canonical C*-dynamical systems}
Looking at Table 1 one can not help  feeling that among the ideals
satisfying $\{0\}\subset  J \subset (\ker\alpha)^\bot$ the
ideal $J=(\ker\alpha)^\bot$ is somewhat privileged. Taking into
account  Cuntz-Pimsner algebras  and their established relation to
crossed products it may seem  completely natural as
$\OO((\ker\alpha)^\bot,X)$ should be considered as 'the smallest'
relative Cuntz-Pimsner algebra containing all the information
about the $C^*$-dynamical system $(\A,\alpha)$. Now, much in the spirit of Katsura's construction, cf. subsection \ref{Katsura's canonical relations}, yet in a slightly different way we will
show that for an arbitrary choice of $J$ the algebra
$\OO(J,X)$ coincides with Cuntz-Pimsner algebra  $\OO(
(\ker\alpha_J)^\bot, X_J)$ where $X_J$ is the
$C^*$-correspondence   of a canonically constructed
$C^*$-dynamical system $(\A_J,\alpha_J)$ which
will be presented below (see Definition~\ref{definition of the
canon}, Theorem \ref{canon}).

\subsection{Unitization of the kernel of an
endomorphism}\label{unit}

Let us fix a $C^*$-dynamical system  $(\A,\alpha)$  and   an ideal $J$  in $\A$.
Above results show us how to  reduce the investigation of crossed products to the case where $J$ is orthogonal to $\ker\al$, thereby let us assume for a while that   $\ker\alpha\cap J=\{0\}$.  The first named author  described in \cite{kwa4} a procedure of extending
$(\A,\alpha)$   up to a  $C^*$-dynamical system $(\A^+,\alpha^+)$ with a property that  the kernel of $\alpha^+$ is unital.
Moreover, the resulting system $(\A^+,\alpha^+)$ is in a sense the smallest  extension of $(\A,\alpha)$
possessing that property, see \cite{kwa4}.

Let us now slightly generalize this construction, which will be essential for  our future purposes.
 Our extension construction depends on the  choice of an ideal
orthogonal to $I:={\ker}\,\alpha$ (recall Definition~\ref{ort}). Thus let us fix a certain ideal $J$ which
satisfies
$$
\{0\} \,\subset \,J\, \subset\,  I^\bot.
$$
By    $\A_J$  we denote the direct sum of quotient algebras
$$
\A_J=\big(\A/I\big) \oplus \big(\A/J\big),
$$
and  we set $\alpha_J:\A_J\to \A_J$ by the formula
\begin{equation}
\label{d_J}
 \A_J\ni \big((a +I)\oplus (b
+J)\big)\stackrel{\alpha_J}{\longrightarrow}(\alpha(a) +I)\oplus
(\alpha(a) +J)\in \A_J.
\end{equation}
 Since $I={\rm ker}\,\alpha$ it follows that an element
$\alpha(a)$ does not depend  on the choice of   a representative
of $a +I$ and so the mapping $\alpha_J$ is well defined.

Note that, as  $I\cap J = \{0\}$,
$$
\alpha_J ([a],[b]) = (0,0)\  \Leftrightarrow  \ \alpha(a)\in I \ \text{and} \ \alpha(a)\in J \ \Leftrightarrow \ \alpha(a)=0.
$$
And therefore
\begin{equation}
\label{kerd_J} \ker \alpha_J = (0,\, \A/J ) \qquad \text{and}\qquad (\ker {\alpha_J})^\bot =(\A/I,\, 0 ).
\end{equation}
Clearly,
$\alpha_J$ is an endomorphism and its  kernel is unital with the
unit of the form $(0 + I) \oplus (1 +J)$. 
Moreover, the $C^*$-algebra $\A$ embeds  into $C^*$-algebra $\A_J$ via
\begin{equation}
\label{A_in} \A \ni a \longmapsto \big(a +I\big)\oplus \big(a +
J\big)\in \A_J.
\end{equation}
 Since $I\cap J = \{0\}$ this mapping is injective  and we will
treat $\A$ as the corresponding subalgebra of $\A_J$. Under this
identification $\alpha_J$ is an extension of $\alpha$, and in particular $\alpha_J (\A_J)=\alpha(\A)\subset \A$.
\smallskip
\begin{rem} The larger ideal  $J$ is the smaller $\A_J$ is.
Namely, since $  {\rm ker}{\alpha_J} = (0,\, \A/J ) $ 
 and
  $J$ varies from $\{0\}$ to
$I^\bot$ it follows that the kernel of $\alpha_J$  varies from $(0,\, \A )$ to
$(0,\,\A/I^{\bot})$. On the other hand, 
 the image of
$\alpha_J$  always concides with $\alpha(\A)\cong\A/I$ 
and thereby  it  does not depend  on the choice of $J$.
The case when $J=I^\bot$ was considered in \cite{kwa4}.
\end{rem}
The next statement presents a  motivation of the preceding construction.
\begin{prop}\label{motivation prop}
Let $(\pi,U,H)$ be a faithful  $J$-covariant representation of $(\A,\alpha)$. 
 Then $J$  is
orthogonal to $I=\ker\alpha$ (cf. Corollary \ref{motivation prop1}) and if $(\A_J,\alpha_J)$ is the extension of $(\A,\alpha)$ constructed above,
then  $\pi$ uniquely extends to the isomorphism $\widetilde{\pi}:\A_J\to C^*(\pi(\A),U^*U)$ such that
$(\widetilde{\pi},U,H)$ is a faithful covariant representation of $(\A_J,\alpha_J)$.
Namely $\widetilde{\pi}$ is given by
\begin{equation}\label{extension of pi to J-algebra}
\widetilde{\pi}(a + I \oplus b + J)=  U^*U \pi(a) + (1- U^*U) \pi(b), \qquad a,b \in \A.
\end{equation}
\end{prop}
\begin{Proof}
If $\widetilde{\pi}:\A_J\to C^*(\pi(\A),U^*U)$ is onto and
$(\widetilde{\pi},U,H)$ is a faithful covariant representation of
$(\A_J,\alpha_J)$, then by  Proposition  \ref{proposition najwazniejsze}   we have
$$
U^*U = \widetilde{\pi}(1 + I \oplus 0 + J).
$$
Thus $\widetilde{\pi}$ is of the form \eqref{extension of pi to
J-algebra} where $\pi=\widetilde{\pi}|_{\A}$ and plainly $(\pi,U,H)$ is a faithful  $J$-covariant representation of $(\A,\alpha)$. Conversely, let $(\pi,U,H)$ be a faithful  $J$-covariant representation of $(\A,\alpha)$.
Then by Corollary \ref{ort-id} $$
\pi( I)= [(1-U^*U)\pi(\A)] \cap \pi(\A), \qquad \pi(J) = [U^*U\pi(\A)]\cap \pi(\A),
 $$
 and consequently, cf. for instance \cite[Lem. 10.1.6]{Kadison}, we have natural isomorphisms
 $$
 \A/ I \cong   U^*U \pi(\A), \qquad   \A/ J \cong [(1-U^*U)\pi(\A)].
 $$
 Thus formula \eqref{extension of pi to
J-algebra} defines an isomorphism $\widetilde{\pi}:\A_J\to C^*(\pi(\A),U^*U)$. It is readily checked, cf. Proposition  \ref{proposition najwazniejsze},  that
$(\widetilde{\pi},U,H)$ is a covariant representation of
$(\A_J,\alpha_J)$.
\end{Proof}

\begin{thm}
Let   $X$ be the $C^*$-correspondence of a $C^*$-dynamical system
$(\A,\alpha)$ and let~$J$ be an ideal in $\A$ orthogonal to the
kernel of $\alpha$. If $X_J$ is  the $C^*$-correspondence of the  system
$({\A_J},{\alpha_J})$ constructed above, then
$$
\OO(J,X)=\OO( (\ker{\alpha_J})^\bot, X_J) = C^*(\A,\alpha, J)=C^*({\A_J},{\alpha_J},(\ker{\alpha_J})^\bot)
$$
is  a universal algebra generated by a copy of the algebra ${\A_J}$ and a partial isometry $u$ subject to relations
\begin{equation}\label{reduced realtions1}
u a u^*={\alpha_J}(a),\,\,\, a\in {\A_J}, \qquad u^*u\in {\A_J}
\end{equation}
(relations \eqref{reduced realtions1} imply that $u^*u$ belongs to the center of ${\A_J}$, cf. Proposition \ref{proposition najwazniejsze}).
\end{thm}
\begin{proof}
In view of Corollary \ref{C-P-cross} we have natural identifications  $\OO(J,X)=C^*(\A,\alpha, J)$ and $\OO( (\ker{\alpha_J})^\bot, X_J) =C^*({\A_J},{\alpha_J},(\ker{\alpha_J})^\bot)$. In order to prove that $C^*(\A,\alpha, J)=C^*({\A_J},{\alpha_J},(\ker{\alpha_J})^\bot)$ it suffices to show that
we have a one-to-one correspondence between faithful  $J$-covariant representations $(\pi,U,H)$ of $(\A,\alpha)$ and faithful $(\ker{\alpha_J})^\bot$-covariant representations $(\widetilde{\pi},U,H)$ of $({\A_J},{\alpha_J})$, but this  follows from Proposition \ref{motivation prop} since by Proposition \ref{proposition najwazniejsze} a faithful covariant representation $(\widetilde{\pi},U,H)$ is $(\ker{\alpha_J})^\bot$-covariant if and only if $U^*U\in \widetilde{\pi}(\A_J)$.
\end{proof}

\subsection{Canonical $C^*$-dynamical systems}
\label{canonic-c*-dynam-syst}

Proposition \ref{reducing C*-Hilbert
bimodules} describes the natural reduction of relations to the case when $J\subset (\ker\alpha)^\bot$. This proposition along with the argument of Subsection \ref{unit} gives us a tool to achieve the goal of the present section; namely, to reduce the whole construction  to the case when $\ker\alpha$ is unital and $J=(\ker\alpha)^\bot$.

\begin{defn}\label{definition of the canon}
Let $(\A,\alpha)$ be a $C^*$-dynamical system and  $J$  an arbitrary ideal in~$\A$. Let $((\A/J_\infty)_{q^{J_\infty}(J)},(\alpha_\infty)_{q^{J_\infty}(J)})$ be the above constructed extension of the reduced $C^*$-dynamical system $(\A/J_\infty,\alpha_\infty)$  given by \eqref{ideal I_infinity}, \eqref{delta_infinity}. We will write
$$
({\A_J},{\alpha_J}):=(\A/J_\infty)_{q^{J_\infty}(J)},(\alpha_\infty)_{q^{J_\infty}(J)})
$$
and say that $({\A_J},{\alpha_J})$ is the \emph{canonical $C^*$-dynamical system} associated with $(\A,\alpha)$ and~$J$.
\end{defn}
\begin{rem}\label{A_J} By Remark~\ref{remark-inf} the above notation does not cause confusion  (in the situation when $
\{0\} \,\subset \,J\, \subset\,  I^\bot $ the pair
$({\A_J},{\alpha_J})$ coincides with the corresponding pair
introduced in \ref{unit}) and therefore we keep the notation
$({\A_J},{\alpha_J})$ in the general situation.
\end{rem}
Combining Propositions \ref{reducing C*-Hilbert bimodules}, \ref{motivation prop}, see also \cite[Cor. 1.7]{kwa4}, we get
\begin{thm}\label{canon}
Let   $X$ be the $C^*$-correspondence of  $(\A,\alpha)$ and let
$J$ be an ideal in $\A$. If $X_J$ is the  $C^*$-correspondence of
the canonical system $({\A_J},{\alpha_J})$, then
$$
\OO(J,X)=\OO( (\ker{\alpha_J})^\bot, X_J) = C^*({\A_J},{\alpha_J},(\ker{\alpha_J})^\bot)
$$
is  a universal algebra generated by a copy of the algebra ${\A_J}$ and a partial isometry $u$ subject to relations
\begin{equation}\label{reduced realtions}
u a u^*={\alpha_J}(a),\,\,\, a\in {\A_J}, \qquad u^*u\in {\A_J}.
\end{equation}
\end{thm}
\begin{rem}\label{6.7}
The usefulness of canonical $C^*$-dynamical system
$(\A_J,\alpha_J)$ manifests in reducing
relations   \eqref{covariance rel1*},
\eqref{covariance rel3}, that apart from endomorphism involve an
ideal and which may degenerate, to the nondegenerated natural
relations \eqref{reduced realtions}. In fact, one could go even
further and use the construction  from \cite{kwa4} to  extend, the
canonical system $(\A_J,\alpha_J)$ up to a $C^*$-dynamical system
$(\B,\tdelta)$ possessing  a complete transfer operator (cf.
subsection \ref{ABL=Kwa-Leb}). Then $\B$ corresponds to the
fixed point subalgebra of $\OO(J,X)$ for the gauge action $\gamma$
(Proposition \ref{RCP algebra}), and by \cite[Prop. 1.9]{kwa3} the
$C^*$-correspondence $\X$ of the $C^*$-dynamical system
$(\B,\tdelta)$ is actually a Hilbert bimodule (in the sense of
Remark~\ref{Hilbbert-bim}). Thus $\OO(J,X)$ can be modeled not
only by the crossed product of $(\B,\tdelta)$, N.3 in Table
\ref{table 1}, cf. Proposition \ref{crossed-ABL1}, but also by
 by the $C^*$-correspondence $\X$, N.4 in
Table \ref{table 2}. Hence the results of \cite{Ant-Bakht-Leb},
\cite{aee} or isomorphism theorem \cite{kwa-ck}
applied to  $(\B,\tdelta)$ can be exploited in the study of
$\OO(J,X)$ in terms of 'Fourier' coefficients.
\end{rem}

We end up by noting that Katsura's 'canonical relations' for
$C^*$-correspondences (see Definitions~\ref{Katsura's
relations},~\ref{Katsura's relations-general}) when applied to the
$C^*$-correspondence $X$   of $(\A,\alpha)$ also leads to a
certain dynamical system, which however in general is slightly
smaller than $(\A_J,\alpha_J)$ and is less natural in our context.
Indeed, by passing if necessary to the reduced objects, we need to
consider only the case when $J\subset (\ker\al)^\bot$, and then
$$
\A_\omega=\{(a,q_J(a'))\in A\oplus A/J: q_{(\ker\al)^\bot}(a)=q_{(\ker\al)^\bot}(a')\}.
$$
In particular, the mapping
$$
\A_\omega\ni (a,q_J(a'))\stackrel{\alpha_\omega}{\longmapsto} (\al(a),q_{J}(\al(a))) \in \A_\omega
$$
yields a well defined endomorphism $\alpha_\omega:\A_\omega\to \A_\omega$, and one sees that $X_\omega$ coincides with the $C^*$-correspondence of the $C^*$-dynamical systems $(\A_\omega,\alpha_\omega)$.
Thus we have three $C^*$-dynamical systems $(\A,\alpha)$,
$(\A_\omega,\alpha_\omega)$, $(\A_J,\alpha_J)$, and  each of them is
an extension of the proceeding one. Indeed,  we have natural
homomorphisms
$$
\A \ni a \stackrel{\iota_1}{\longmapsto} a\oplus q_J(a) \in \A_\omega, \qquad \A_\omega \ni (a,q_J(a')) \stackrel{\iota_2}{\longmapsto}  q_{\ker\al}(a)\oplus q_J(a') \in \A_J.
$$
Clearly, $\iota_1$ is injective  and to see that $\iota_2$ is injective note that $\iota_2(a,q_J(a'))=0$ means that $a\in\ker\alpha$ and $a'\in J\subset (\ker\al)^\bot $, and then  relation $q_{(\ker\al)^\bot}(a)=q_{(\ker\al)^\bot}(a')=0$ imply that $a=0$. The  monomorphisms $\iota_1, \iota_2 $ make the following diagram commute
$$
\begin{xy}
\xymatrix{\A \ar[d]_{\alpha} \ar[r]^{\iota_1} & \A_\omega\ar[d]_{\alpha_\omega} \ar[r]^{\iota_2} &  \A_J \ar[d]^{\al_J}
       \\
    \A \ar[r]_{\iota_1} & \A_\omega  \ar[r]_{\iota_2 }&  \A_J
              }
  \end{xy}.
$$
Moreover, $\iota_1$ is an isomorphism iff $J=(\ker\al)^\bot$ and $\iota_2$ is    an isomorphism iff $(\ker\al)^\bot$ is unital. In particular, in 'Katsura's picture',
 the crossed product $C^*(\A,\alpha, J)$ could be considered as a universal $C^*$-algebra subject to relations
 $$
ua u^*=\alpha_\omega(a),\quad   a \in \A_\omega,\qquad
\{a\in \A_\omega: u^*u a=a\}=(\ker\al_\omega)^\bot,
$$
which apparently are more complicated than relations \eqref{reduced realtions}.

\textbf{B. K.  Kwa\'sniewski}, 
 \noindent \textsc{Institute of Mathematics,  University  of Bialystok},
\textsc{ ul. Akademicka 2, PL-15-267  Bialystok, Poland }\\
 \emph{e-mail:} \texttt{bartoszk@math.uwb.edu.pl}, \\
 \emph{www:} \texttt{http://math.uwb.edu.pl/$\sim$zaf/kwasniewski}
 \bigskip
 
 \noindent \textbf{A. V. Lebedev},  
\noindent \textsc{Department of Mechanics and Mathematics,
Belarus State University, 
pr. Nezavisimosti, 4, 220050, Minsk, Belarus}, $\&$ \\
\textsc{Institute of Mathematics,  University  of Bialystok},
\textsc{ ul. Akademicka 2, PL-15-267  Bialystok, Poland},
 \emph{e-mail:} \texttt{lebedev@bsu.by}
 \end{document}